\documentclass[11pt,a4paper,reqno]{amsart}
\usepackage{amssymb,amsfonts,euscript,amscd,amsmath,amsthm,amsopn}
\usepackage{color, graphicx, bbm, url}
\usepackage[usenames,dvipsnames]{xcolor}
\usepackage[utf8]{inputenc}
\usepackage[position=top]{subfig}
\usepackage{standalone}
\usepackage{xspace}
\usepackage{gensymb}
\usepackage{booktabs,tabularx}
\usepackage{soul}

\usepackage{tikz}
\usetikzlibrary{calc}

\makeatletter
\tikzset{
  xyz frame/.code n args={3}{%
    \begingroup
    \tikz@scan@one@point\pgfutil@firstofone#1\relax
    \pgf@xa=\pgf@x
    \pgf@ya=\pgf@y
    \tikz@scan@one@point\pgfutil@firstofone#2\relax
    \pgf@xb=\pgf@x
    \pgf@yb=\pgf@y
    \tikz@scan@one@point\pgfutil@firstofone#3\relax
    \edef\tikz@marshall{\noexpand\endgroup\noexpand\pgfsetxvec{\noexpand\pgfpoint{\the\pgf@xa}{\the\pgf@ya}}%
      \noexpand\pgfsetyvec{\noexpand\pgfpoint{\the\pgf@xb}{\the\pgf@yb}}%
      \noexpand\pgfsetzvec{\noexpand\pgfpoint{\the\pgf@x}{\the\pgf@y}}}%
    \tikz@marshall
  },
  on layer/.code={
    \pgfonlayer{#1}\begingroup
    \aftergroup\endpgfonlayer
    \aftergroup\endgroup
  },
}
\makeatother

\makeatletter
\newcommand*{\inlineequation}[2][]{%
  \begingroup
    \refstepcounter{equation}%
    \ifx\\#1\\%
    \else
      \label{#1}%
    \fi
    \relpenalty=10000 %
    \binoppenalty=10000 %
    \ensuremath{%
      #2%
    }%
    ~\@eqnnum
  \endgroup
}
\makeatother

\pgfdeclarelayer{behind}
\pgfdeclarelayer{in front}
\pgfsetlayers{behind,main,in front}

\newcommand{\kay}{\textit{k}}

\usepackage{enumitem}
\newlist{MyIndentedList}{itemize}{4}
\setlist[MyIndentedList,1]{%
    label={},
    noitemsep,
    leftmargin=0pt,
    }
\setlist[MyIndentedList]{%
    label={},
    noitemsep,
    }

\newcommand{\niton}{\not\owns}

\usepackage[bookmarks, bookmarksdepth=2, colorlinks=true, linkcolor=blue, citecolor=blue, urlcolor=blue]{hyperref}
\newcommand{\arxiv}[1]{\href{http://arxiv.org/abs/#1}{{\tt arXiv:#1}}}

 \newcommand\blfootnote[1]{%
  \begingroup
  \renewcommand\thefootnote{}\footnote{#1}%
  \addtocounter{footnote}{-1}%
  \endgroup
  }

\newcommand{\QQ}{\mathbb{Q}}
\newcommand{\RR}{\mathbb{R}}
\newcommand{\CC}{\mathbb{C}}
\newcommand{\ZZ}{\mathbb{Z}}
\newcommand{\PP}{\mathbb{P}}

\theoremstyle{plain}
\newtheorem{thm}{Theorem}
\newtheorem{lem}[thm]{Lemma}
\newtheorem{prop}[thm]{Proposition}

\newtheorem{defn}[thm]{Definition}

\newtheorem{defnprop}[thm]{Definition/Proposition}

\theoremstyle{definition}
\newtheorem{exmp}[thm]{Example}

\theoremstyle{remark}
\newtheorem*{rem}{Remark}

\title{Minimal Problems for the \linebreak Calibrated Trifocal Variety}
\author{Joe Kileel}

\begin{document} 
 
\begin{abstract} \noindent
We determine the algebraic degree 
of minimal problems
for the calibrated trifocal variety 
in computer vision.  We rely on
numerical algebraic geometry
and the homotopy continuation 
software \texttt{Bertini}.

\end{abstract}

\maketitle

\section{Introduction} \label{sec:1}

\blfootnote{\textit{2010 Mathematics Subject Classification.} 14M20, 14Q15, 14N99, 15A69, 65H20, 68T45.}
\blfootnote{\textit{Key words and phrases.} calibrated trifocal variety, minimal problems, numerical algebraic geometry.} 

In computer vision, one fundamental task is \textit{3D reconstruction}: the
recovery of three-dimensional scene geometry
from two-dimensional images.   
In 1981, Fischler and Bolles proposed a methodology 
for \textit{3D} reconstruction that is robust to outliers in image data \cite{FB}.
This is known as Random Sampling Consensus (RANSAC) and 
it is a paradigm in vision today \cite{ASSSS}.
RANSAC consists of three steps.
To compute a piece of the \textit{3D} scene:

\begin{itemize}
\item Points, lines and other features
that are images of the same source are detected in the photos.
These matches are the \textit{image data}.   
\item  A minimal sample of image data is randomly
selected.  \textit{Minimal} means that only a positive finite number of \textit{3D} geometries are exactly consistent with the sample. Those \textit{3D} geometries are computed.
\item  To each computed
\textit{3D} geometry, the rest of the image data is compared.  If one is approximately consistent with enough of 
the image data, it is kept.  Else, the second step is repeated with a new sample.
\end{itemize}

Computing the finitely many \textit{3D} geometries in the second step is called a \textit{minimal
problem}.  
Typically, it is done by solving a corresponding zero-dimensional polynomial
system, with coefficients that are functions of the sampled image data \cite{Kuk}.
Since this step is carried out
thousands of times in a full reconstruction, it is necessary to 
design efficient, specialized solvers.  
One of the most used \textit{minimal solvers} in vision is Nist\'{e}r's \cite{Nis}, based on Gr\"{o}bner bases, 
to recover the relative 
position of two calibrated cameras.

The concern of this paper is the recovery of the relative
position of \textit{three} calibrated cameras from image data.
To our knowledge, no satisfactory solution to this basic problem
exists in the literature.  
Our main result is the determination of 
the \textit{algebraic degree}
of $66$ minimal problems for the recovery of three calibrated cameras;
in other words, we find the generic number of complex solutions (see Theorem \ref{thm:main}).   
The solution sets for particular random instances 
are available at this project's computational webpage:

\medskip
{\small \url{https://math.berkeley.edu/~jkileel/CalibratedMinimalProblems.html}}.
\medskip

As a by-product, we can derive minimal solvers
for each case.  
Our techniques come from \textit{numerical algebraic geometry} \cite{SW},
and we rely on the homotopy continuation software \texttt{Bertini} \cite{BHSWSoftware}.  
This implies that our results are correct only with very high probability; in ideal arithmetic, with
probability $1$.  Mathematically, the main object in this paper
is a particular projective algebraic variety $\mathcal{T}_\text{cal}$, 
which is a convenient moduli space for the relative position of three calibrated cameras.
This variety is $11$-dimensional, degree $4912$ inside the projective space $\PP^{26}$ of $3 \times 3 \times 3$ tensors (see Theorem \ref{thm:4912}).
We call it the \textit{calibrated trifocal variety}.
Theorem \ref{thm:codim} formulates our minimal problems as slicing $\mathcal{T}_\text{cal}$ by special linear subspaces of $\PP^{26}$.

The rest of this paper is organized as follows.  In Section \ref{sec:2}, we make our minimal problems mathematically precise
and we state Theorem \ref{thm:main}.  
In Section \ref{sec:3}, we examine image correspondences using multi-view varieties and then trifocal tensors \cite[Chapter 15]{HZ}.
In Section \ref{sec:4}, we prove that trifocal tensors and camera configurations are equivalent.
In Section \ref{sec:5}, we introduce the calibrated trifocal variety $\mathcal{T}_\text{cal}$ and prove several useful facts.
Finally, in Section \ref{sec:6},
we present a computational
proof of the main result Theorem \ref{thm:main}.

\bigskip

{\footnotesize{\noindent \textbf{Acknowledgements.}  I thank my advisor Bernd Sturmfels.  
I also thank Jonathan Hauenstein for
help with \texttt{Bertini} and numerical algebraic geometry. I benefitted from technical conversations
with Justin Chen, Luke Oeding, Kristian Ranestad and Jose Rodriguez. 
}}

\section{Statement of Main Result} \label{sec:2}

We begin by giving several definitions.  Throughout this paper,
we work with the standard camera model of the projective camera \cite[Section 6.2]{HZ}.

\begin{defn}
A (projective) \textup{\textbf{camera}} is a full rank $3 \times 4$ matrix in $\CC^{3 \times 4}$ defined up to multiplication by a nonzero scalar.
\end{defn}

Thus, a camera corresponds to a linear projection $\PP^3 \dashrightarrow \PP^2$.
The \textbf{center} of a camera $A$ is the point $\textup{ker}(A) \in \PP^3$.  A camera is \textbf{real} if $A \in \RR^{3 \times 4}$.

\begin{defn}
A \textup{\textbf{calibrated camera}} is a $3 \times 4$ matrix in $\CC^{3 \times 4}$ whose left $3 \times 3$ submatrix is in the special orthogonal group $\textup{SO}(3, \CC)$.
\end{defn}

Real calibrated cameras have the interpretation of cameras with known and normalized 
\textit{internal parameters} (e.g. focal length) \cite[Subsection 6.2.4]{HZ}.  
In practical situations, this information can be available during \textit{3D} reconstruction.
Note that calibration of a camera is preserved by right multiplication by elements of the following subgroup of $\textup{GL}(4,\CC)$:
$$
\mathcal{G} := \{g \in \CC^{4 \times 4} \, | \, (g_{ij})_{\tiny{{1 \leq i, j \leq 3}}} \in \textup{SO}(3, \CC), \, g_{41} = g_{42} = g_{43} = 0 \text{ and } g_{44} \neq 0\}.
$$

\noindent Elements in $\mathcal{G}$ act on $\mathbb{A}^3 \subset \PP^3$ as composites of rotations, translations and central dilations.  
In the calibrated case of \textit{3D} reconstruction, one aims to recover camera positions (and afterwards
the \textit{3D} scene) up to those motions, since recovery of absolute positions is not possible from image data alone.

\begin{defn}
A \textup{\textbf{configuration}} of three calibrated cameras is an orbit of the action of the group $\mathcal{G}$ above
on the set:
$$ \{(A, B, C) \, | \, A, B, C \textup{ are calibrated cameras} \} $$
via simultaneous right multiplication.
\end{defn}

By abuse of notation, we will call $(A,B,C)$ a calibrated camera configuration, 
instead of always denoting the orbit containing $(A,B,C)$.

As mentioned in Section \ref{sec:1}, the image data used in \textit{3D} reconstruction 
typically are points and lines in the photos that match.  This is made precise as follows.
Call elements of $\PP^2$ \textbf{image points}, and elements of the dual projective plane
$(\PP^2)^{\vee}$ \textbf{image lines}.
An element of $(\PP^2 \sqcup (\PP^2)^{\vee})^{\times 3}$ is a \textbf{point/line image correspondence}.
For example, an element of $\PP^2 \times \PP^2 \times (\PP^2)^{\vee}$ is called a point-point-line image correspondence, denoted \textit{PPL}.

\begin{defn}\label{defn:consistent}
A calibrated camera configuration $(A,B,C)$ is \textup{\textbf{consistent}} with a given point/line image correspondence
if there exist a point in $\PP^3$ and a line in $\PP^3$ containing it such that are such that $(A,B,C)$ respectively map these to the given points and lines in $\PP^2$.  
\end{defn}

For example, explicitly, a configuration $(A,B,C)$ is consistent with 
a given point-point-line image correspondence $(x, x', \ell'') \in \PP^2 \times \PP^2 \times (\PP^2)^{\vee}$ if  
there exist $(X, L) \in \PP^3 \times \textup{Gr}(\PP^1, \PP^3)$ with $X \in L$
such that $AX = x, BX = x',$ and $CL = \ell''$.  In particular, this implies that
$X \neq \textup{ker}(A), \textup{ker}(B)$ and $\textup{ker}(C) \notin L $.
We say that a configuration $(A,B,C)$ is consistent with a set of point/line correspondences
if it is consistent with each correspondence.

We give a numerical example to illustrate Theorem \ref{thm:main} on the next page:

\begin{exmp}  \label{exmp:160}
Given the following set of real, random correspondences:\footnote{\label{footnote:precision}For ease of presentation, double precision floating point numbers are truncated here.}
{\footnotesize{
\begin{align*}
&PPP:  \,\, \begin{bmatrix} 0.6132 \\ 0.8549 \\ 0.5979 \end{bmatrix}, \, \begin{bmatrix} 0.4599 \\ 0.5713 \\ 0.1812 \end{bmatrix}, \, \begin{bmatrix} 0.6863 \\ 0.4508 \\ 0.1834 \end{bmatrix} \,\,\,\,\,
& \,\,\,\,\, PPL:  \,\, \begin{bmatrix} 0.6251 \\ 0.9248 \\ 0.9849 \end{bmatrix}, \, \begin{bmatrix} 0.3232 \\ 0.5453 \\ 0.6941 \end{bmatrix}, \, \begin{bmatrix} 0.3646 \\ 0.1497 \\ 0.1364 \end{bmatrix} \\
&PPL: \,\, \begin{bmatrix} 0.4970 \\ 0.6532 \\ 0.8429 \end{bmatrix}, \, \begin{bmatrix} 0.5405 \\ 0.8342 \\ 0.6734 \end{bmatrix}, \, \begin{bmatrix} 0.2692 \\ 0.8861 \\ 0.1333 \end{bmatrix} \,\,\,\,\, 
& \,\,\,\,\, PPL: \,\, \begin{bmatrix} 0.2896 \\ 0.6909 \\ 0.4914 \end{bmatrix}, \, \begin{bmatrix} 0.6898 \\ 0.9855 \\ 0.6777 \end{bmatrix}, \, \begin{bmatrix} 0.6519 \\ 0.8469 \\ 0.6855 \end{bmatrix} \\
&PPL: \,\, \begin{bmatrix} 0.8933 \\ 0.3375 \\ 0.1054 \end{bmatrix}, \, \begin{bmatrix} 0.7062 \\ 0.6669 \\ 0.7141 \end{bmatrix}, \, \begin{bmatrix} 0.3328 \\ 0.8228 \\ 0.6781 \end{bmatrix}. &
\end{align*}
}}
\noindent In the notation of Theorem \ref{thm:main}, this is a generic instance of the minimal problem `$1 PPP + 4 \textit{PPL}$'.  
Up to the action of $\mathcal{G}$, there are only a positive finite number of three calibrated
cameras
that are exactly consistent with this image data, namely 160 complex configurations.   
For this instance, it turns out that 18 of those configurations are real.  For example,
one is:
\vspace{-0.3cm}

$$
{\footnotesize{A = }}{\tiny{\begin{bmatrix} 1 & 0 & 0 & 0 \\ 0 & 1 & 0 & 0 \\ 0 & 0 & 1 & 0 \end{bmatrix}}}, \,
{\footnotesize{B = }}{\tiny{\begin{bmatrix} -0.22 & 0.95 & -0.18 & 1 \\ 0.96 & 0.24 & 0.08 & 1.44 \\ -0.12 & 0.15 & 0.97 & 0.97 \end{bmatrix}}}, \,
{\footnotesize{C = }}{\tiny{\begin{bmatrix} 0.17 & 0.94 & -0.28 & 1.41 \\ -0.95 & 0.22 & 0.18 & -0.13 \\ -0.24 & -0.23 & -0.94 & -1.16 \end{bmatrix}}}. 
\vspace{0.15cm}
$$
In a RANSAC run for \textit{3D} reconstruction, the image data above is identified by feature detection 
software such as \texttt{SIFT} \cite{Low}.  Also, only the real configurations are compared for agreement 
with further image data.
\end{exmp}

In Example \ref{exmp:160} above, 160 is the \textbf{algebraic degree} of the minimal problem `$1 PPP + 4 \textit{PPL}$'.
This means that for correspondences in a nonempty Zariski open (hence measure 1) subset of 
$(\PP^2 \times \PP^2 \times \PP^2) \times (\PP^2 \times \PP^2 \times (\PP^2)^{\vee})^{\times 4}$,
there are 160 consistent complex configurations.  Given generic real correspondences,
the number of real configurations varies, but 160 is an upper bound.  

The cases in Theorem \ref{thm:main} admit a uniform treatment that we give below.

\newpage 

\begin{thm} \label{thm:main}
The rows of the following table display the algebraic degree for $66$ minimal problems
across three calibrated views.
Given generic point/line image correspondences
in the amount specified by the entries in the first five columns,
then the number of calibrated camera configurations over $\CC$ 
that are consistent with those correspondences equals the entry in the sixth column. 

{\tiny{
\begin{table}[h!]
\fontsize{6.65}{7.65}\selectfont
\centering
\begin{tabular}{||c | c | c | c | c || c||} 
 \hline
$\#$PPP & $\#$PPL & $\#$PLP & $\#$LLL & $\#$PLL & $\#$configurations\\ [0.5ex] 
 \hline\hline
3 & 1 & 0 & 0 & 0 & 272 \\
\hline
3 & 0 & 0 & 1 & 0 & 216 \\
\hline
3 & 0 & 0 & 0 & 2 & 448 \\
\hline 
2 & 2 & 0 & 0 & 1 & 424  \\
\hline 
2 & 1 & 1 & 0 & 1 & 528 \\
\hline 
2 & 1 & 0 & 1 & 1 & 424 \\
\hline 
2 & 1 & 0 & 0 & 3 & 736 \\
\hline 
2 & 0 & 0 & 2 & 1 & 304 \\
\hline 
2 & 0 & 0 & 1 & 3 & 648 \\
\hline
2 & 0 & 0 & 0 & 5 & 1072 \\
\hline
1 & 4 & 0 & 0 & 0 & 160 \\
\hline
1 & 3 & 1 & 0 & 0 & 520 \\
\hline
1 & 3 & 0 & 1 & 0 & 360 \\
\hline
1 & 3 & 0 & 0 & 2 & 520 \\
\hline
1 & 2 & 2 & 0 & 0 & 672 \\
\hline
1 & 2 & 1 & 1 & 0 & 552 \\
\hline
1 & 2 & 1 & 0 & 2 & 912 \\
\hline
1 & 2 & 0 & 2 & 0 & 408 \\
\hline
1 & 2 & 0 & 1 & 2 & 704 \\
\hline
1 & 2 & 0 & 0 & 4 & 1040 \\
 \hline
 1 & 1 & 1 & 2 & 0 & 496 \\
 \hline
 1 & 1 & 1 & 1 & 2 & 896 \\
 \hline
 1 & 1 & 1 & 0 & 4 & 1344 \\
 \hline
 1 & 1 & 0 & 3 & 0 & 368 \\
 \hline
 1 & 1 & 0 & 2 & 2 & 736 \\
 \hline
 1 & 1 & 0 & 1 & 4 & 1184 \\
 \hline
 1 & 1 & 0 & 0 & 6 & 1672 \\
 \hline
 1 & 0 & 0 & 4 & 0 & 360 \\
 \hline
 1 & 0 & 0 & 3 & 2 & 696 \\
 \hline
 1 & 0 & 0 & 2 & 4 & 1176 \\
 \hline
 1 & 0 & 0 & 1 & 6 & 1680 \\
 \hline
 1 & 0 & 0 & 0 & 8 & 2272 \\
 \hline
 0 & 5 & 0 & 0 & 1 & 160 \\
 \hline
 0 & 4 & 1 & 0 & 1 & 616 \\
 \hline
 0 & 4 & 0 & 1 & 1 & 456 \\
 \hline
 0 & 4 & 0 & 0 & 3 & 616 \\
 \hline
 0 & 3 & 2 & 0 & 1 & 1152 \\
 \hline
 0 & 3 & 1 & 1 & 1 & 880 \\
 \hline
 0 & 3 & 1 & 0 & 3 & 1280 \\
 \hline
 0 & 3 & 0 & 2 & 1 & 672 \\
 \hline
 0 & 3 & 0 & 1 & 3 & 1008 \\
 \hline
 0 & 3 & 0 & 0 & 5 & 1408 \\
 \hline
 0 & 2 & 2 & 1 & 1 & 1168 \\
 \hline
 0 & 2 & 2 & 0 & 3 & 1680 \\
 \hline
 0 & 2 & 1 & 2 & 1 & 1032 \\
 \hline
 0 & 2 & 1 & 1 & 3 & 1520 \\
 \hline
 0 & 2 & 1 & 0 & 5 & 2072 \\
 \hline
 0 & 2 & 0 & 3 & 1 & 800 \\
 \hline
 0 & 2 & 0 & 2 & 3 & 1296 \\
 \hline 
 0  & 2 & 0 & 1 & 5 & 1848 \\
 \hline
 0 & 2 & 0 & 0 & 7 & 2464 \\
 \hline
 0 & 1 & 1 & 3 & 1 & 1016 \\
 \hline
 0 & 1 & 1 & 2 & 3 & 1552 \\
 \hline
 0 & 1 & 1 & 1 & 5 & 2144 \\
 \hline
 0 & 1 & 1 & 0 & 7 & 2800 \\
 \hline
 0 & 1 & 0 & 4 & 1 & 912 \\
 \hline
 0 & 1 & 0 & 3 & 3 & 1456 \\
 \hline
 0 & 1 & 0 & 2 & 5 & 2088 \\
 \hline
 0 & 1 & 0 & 1 & 7 & 2808 \\
 \hline
 0 & 1 & 0 & 0 & 9 & 3592 \\
 \hline
 0 & 0 & 0 & 5 & 1 & 920 \\
 \hline
 0 & 0 & 0 & 4 & 3 & 1464 \\
 \hline
 0 & 0 & 0 & 3 & 5 & 2176 \\
 \hline
 0 & 0 & 0 & 2 & 7 & 3024 \\
 \hline
 0 & 0 & 0 & 1 & 9 & 3936 \\
 \hline
 0 & 0 & 0 & 0 & 11 & 4912 \\
 \hline
\end{tabular}
\end{table}
}}
\end{thm}

\begin{rem}
A calibrated camera configuration $(A,B,C)$ has 11 degrees of freedom (Theorem \ref{thm:4912}),
and the first five columns in the table above represent conditions of 
codimension 3, 2, 2, 2, 1, respectively \,\!(Theorem \ref{thm:codim}).
\end{rem}

\begin{rem}
The algebraic degrees in Theorem \ref{thm:main} are \ul{intrinsic to the underlying camera geometry}.
However, our method of proof uses a device from multi-view geometry called trifocal tensors, which 
breaks symmetry between $(A,B,C)$.  There are other minimal problems for three
calibrated views involving image correspondences of type `\textit{LPP}', `\textit{LPL}', `\textit{LLP}'.
These also possess intrinsic algebraic degrees; but they are not covered by the 
non-symmetric proof technique used here.
\end{rem}

\section{Correspondences} \label{sec:3}

In this section, we examine point/line image correspondences. 
In the first part, we use \textit{multi-view varieties} to describe correspondences. 
This approach furnishes exact polynomial systems for the minimal problems
in Theorem \ref{thm:main}.  However, each parametrized system has a different structure (in terms 
of number and degrees of equations).  This would force a direct analysis for Theorem \ref{thm:main} to proceed case-by-case, and
moreover, each system so obtained is computationally unwieldy.
In Subsection \ref{subsec:3.2}, we recall the construction of the \textit{trifocal tensor} \cite[Chapter 15]{HZ}. 
This is a point $T_{A,B,C} \in \CC^{3 \times 3 \times 3}$ associated to cameras $(A,B,C)$.  It encodes
necessary conditions for $(A,B,C)$ to be consistent with different types of correspondences.  Tractable relaxations to the minimal
problems in Theorem \ref{thm:main} are thus obtained, each with similar structure.  We emphasize that everything in Section \ref{sec:3}
applies equally to calibrated cameras $(A,B,C)$ as well as to uncalibrated cameras.

\subsection{Multi-view varieties}\label{subsec:3.1}
Let $A,B,C \in \CC^{3 \times 4}$ be three projective cameras, not necessarily calibrated.  
Denote by $\alpha: \PP^3 \dashrightarrow \PP^{2}_{A}$, $\beta: \PP^3 \dashrightarrow \PP^{2}_{B}$,
$\gamma: \PP^3 \dashrightarrow \PP^{2}_{C}$ the corresponding linear projections.  We make:

\begin{defn} \label{defn:multi-view}
Fix projective cameras $A,B,C$ as above.  Denote by $\mathcal{F\ell}_{0,1}$ the incidence
variety ${\big{\{}}(X,L) \in \PP^{3} \times \textup{Gr}(\PP^{1}, \PP^{3}) \,\, \big | \,\, X \in L {\big{\}}}$. Then the\textup{:}
\vspace{0.3em}
\begin{itemize}
\setlength\itemsep{0.7em}
\item \textup{\textbf{PLL multi-view variety}} denoted $X_{A,B,C}^{PLL}$ is the closure of the image of
$\mathcal{F\ell}_{0,1} \dashrightarrow \PP_{A}^{2} \times (\PP_{B}^{2})^{\vee} \times (\PP_{C}^{2})^{\vee}, \,\, (X, L) \mapsto {\big{(}}\alpha(X), \beta(L), \gamma(L){\big{)}}$
\item \textup{\textbf{LLL multi-view variety}} denoted $X_{A,B,C}^{LLL}$ is the closure of the image of
$ \textup{Gr}(\PP^{1}, \PP^{3}) \dashrightarrow (\PP_{A}^{2})^{\vee} \times (\PP_{B}^{2})^{\vee} \times (\PP_{C}^{2})^{\vee}, \, L \mapsto {\big{(}}\alpha(L), \beta(L), \gamma(L){\big{)}}$
\item \textup{\textbf{PPL multi-view variety}} denoted $X_{A,B,C}^{PPL}$ is the closure of the image of
$\mathcal{F\ell}_{0,1} \dashrightarrow \PP_{A}^{2} \times \PP_{B}^{2} \times (\PP_{C}^{2})^{\vee}, \,\, (X, L) \mapsto {\big{(}}\alpha(X), \beta(X), \gamma(L){\big{)}}$
\item \textup{\textbf{PLP multi-view variety}} denoted $X_{A,B,C}^{PLP}$ is the closure of the image of
$\mathcal{F\ell}_{0,1} \dashrightarrow \PP_{A}^{2} \times (\PP_{B}^{2})^{\vee} \times \PP_{C}^{2}, \,\, (X, L) \mapsto {\big{(}}\alpha(X), \beta(L), \gamma(X){\big{)}}$
\item \textup{\textbf{PPP multi-view variety}} denoted $X_{A,B,C}^{PPP}$ is the closure of the image of
$\PP^{3} \dashrightarrow \PP_{A}^{2} \times \PP_{B}^{2} \times \PP_{C}^{2}, \,\, X \mapsto {\big{(}}\alpha(X), \beta(X), \gamma(X){\big{)}}$.
\end{itemize}
\end{defn}

\vspace{0.05cm}
Next, we give the dimension and equations for these multi-view varieties; the `\textit{PPP}' case has appeared in \cite{AST}.  
In the following, we notate
$x \in \PP^{2}_{A}$, $x' \in \PP^{2}_{B}$, $x'' \in \PP^{2}_{C}$ for image points and 
$\ell \in (\PP^{2}_{A})^{\vee}$,  $\ell' \in (\PP^{2}_{B})^{\vee}$, $\ell'' \in (\PP^{2}_{C})^{\vee}$ for image lines.  
Also, we postpone treatment of the `\textit{PLL}' case  to Subsection \ref{subsec:3.2}.
In particular, the trilinear form $T_{A,B,C}(x, \ell', \ell'')$ will be defined there.

\begin{thm} \label{thm:multi-view}
Fix $A,B,C$. The multi-view varieties from
Definition \textup{\ref{defn:multi-view}} are irreducible.  If $A,B,C$ have linearly independent centers in $\PP^{3}$, then the varieties have the
following dimensions and multi-homogeneous prime ideals.

\vspace{0.2em}

\begin{itemize}
\setlength\itemsep{0.95em}
\item $\textup{dim}(X_{A,B,C}^{PLL}) = 5$ and $I(X_{A,B,C}^{PLL}) = \langle T_{A,B,C}(x, \ell', \ell'') \rangle \subset \CC[x_{i}, \ell'_{j}, \ell''_{k}] $
\item $\textup{dim}(X_{A,B,C}^{LLL}) = 4$ and $I(X_{A,B,C}^{LLL}) \subset \CC[\ell_{i}, \ell'_{j}, \ell''_{k}]$ is generated by the maximal minors of
the matrix $\begin{pmatrix} A^{T} \ell & B^{T} \ell' & C^{T} \ell'' \end{pmatrix}_{4 \times 3}$
\item $\textup{dim}(X_{A,B,C}^{PPL}) = 4$ and $I(X_{A,B,C}^{PPL}) \subset \CC[x_{i}, x'_{j}, \ell''_{k}]$ is generated by the maximal minors of
the matrix $\begin{pmatrix} A & x & 0 \\ B & 0 & x' \\ \ell''^{T}C & 0 & 0 \end{pmatrix}_{7 \times 6}$
\item $\textup{dim}(X_{A,B,C}^{PLP}) = 4$ and $I(X_{A,B,C}^{PLP}) \subset \CC[x_{i}, \ell'_{j}, x''_{k}]$ is generated by the maximal minors of
the matrix $\begin{pmatrix} A & x & 0 \\ C & 0 & x'' \\ \ell'^{T}B & 0 & 0 \end{pmatrix}_{7 \times 6}$
\item $\textup{dim}(X_{A,B,C}^{PPP}) = 3$ and $I(X_{A,B,C}^{PPP}) \subset \CC[x_{i}, x'_{j}, x''_{k}]$ is generated by the maximal minors of
the matrix $\begin{pmatrix} A & x & 0 & 0 \\ B & 0 & x' & 0 \\ C & 0 & 0 & x'' \end{pmatrix}_{9 \times 7}$ together with 
$\textup{det} \begin{pmatrix} A & x & 0 \\ B & 0 & x' \end{pmatrix}_{6 \times 6}$
and \,\, $\textup{det} \begin{pmatrix} A & x & 0 \\ C & 0 & x'' \end{pmatrix}_{6 \times 6}$
and \,\, $\textup{det} \begin{pmatrix} B & x' & 0 \\ C & 0 & x'' \end{pmatrix}_{6 \times 6}$
\end{itemize}
\end{thm}
\begin{proof}
Irreducibility is clear from Definition \ref{defn:multi-view}.  For the dimension and prime ideal statements,
we may assume that: 
\vspace{0.1em}
$$A = \begin{bmatrix} 1 & 0 & 0 & 0 \\ 0 & 1 & 0 & 0 \\ 0 & 0 & 1 & 0 \end{bmatrix}\!, \,\,
B = \begin{bmatrix} 1 & 0 & 0 & 0 \\ 0 & 1 & 0 & 0 \\ 0 & 0 & 0 & 1 \end{bmatrix}\!, \,\,
C = \begin{bmatrix} 1 & 0 & 0 & 0 \\ 0 & 0 & 1 & 0 \\ 0 & 0 & 0 & 1 \end{bmatrix}.$$

\noindent This is without loss of generality in light of the following group symmetries.  
Let $g, g', g'' \in \textup{SL}(3, \CC)$ and $h \in \textup{SL}(4, \CC)$.  To illustrate, 
consider the third case above, and let $J_{A,B,C}^{PPL}  \subset \CC[x_{i}, x'_{j}, \ell''_{k}] $ be the 
ideal generated by the maximal minors  mentioned there.  It is straightforward to check that:

\vspace{-0.3em}

$$ I(X_{Ah,Bh,Ch}^{PPL}) = I(X_{A,B,C}^{PPL}) \,\,\, \textup{ and } \,\,\, J_{Ah,Bh,Ch}^{PPL} = J_{A,B,C}^{PPL}.$$

\vspace{0.6em}

\noindent Also, we can check that:

\vspace{-0.3em}

$$ I(X_{gA,\,\, g'B,\,\, g''C}^{PPL}) = (g, g', \wedge^{2} g'') \cdot I(X_{A,B,C}^{PPL})$$ 
$$ \textup{ and } \,\,\,\, J_{gA,\,\, g'B,\,\, g''C}^{PPL} = (g, g', (g''^{\,T})^{-1}) \cdot J_{A,B,C}^{PPL}.$$

\vspace{0.4em}

\noindent Here the left, linear action of $\textup{SL}(3, \CC) \times \textup{SL}(3, \CC) \times \textup{SL}(3, \CC)$ on $\CC[x_{i}, x'_{j}, \ell''_{k}]$ is via 
$(g, g', g'') \cdot f(x, x', \ell'') = f(g^{-1}x, g'^{-1}x', g''^{-1}\ell'')$ for $f \in \CC[x_{i}, x'_{j}, \ell''_{k}]$.
Also, $\wedge^{2} g'' = (g''^{\,T})^{-1} \in \CC^{3 \times 3}$.  So, for the `\textit{PPL}' case, $I$ and $J$ transform in the same way when $(A,B,C)$ is replaced by
$(gAh, g'Bh, g''Ch)$; in the other cases, this holds similarly.  Assuming that $A, B, C$ have linearly independent centers, we may choose $g, g', g'', h$ to harmlessly move
the cameras into the position above.  Now using the computer algebra system \texttt{Macaulay2} \cite{GS}, we
verify the dimension and prime ideal statements for this special position.
\end{proof}

\begin{rem}
In Theorem \ref{thm:multi-view}, if $A, B, C$ do not have linearly independent centers,
then the minors described still vanish on the multi-view varieties, by continuity in $(A,B,C)$.
\end{rem}

Now, certainly a point/line correspondence that is consistent with $(A,B,C)$
lies in the appropriate multi-view variety; consistency means
that the correspondence is a point in the \textit{set-theoretic} image of
the appropriate rational map in Definition \ref{defn:multi-view}.  
Since the multi-view varieties
are the Zariski closures of those set-theoretic images, care is needed to make a converse.  We require:

\begin{defn} \label{defn:epi}
Let $A,B,C$ be three projective cameras with distinct centers.  The \textup{\textbf{epipole}} denoted $\textbf{e}_{1 \leftarrow 2}$ is the point
$\alpha(\textup{ker}(B)) \in \PP^2_{A}$.  That is, $\textbf{e}_{1 \leftarrow 2}$ is the image under $A$ of the center of $B$.  
Epipoles $\textbf{e}_{1 \leftarrow 3}, \textbf{e}_{2 \leftarrow 1}, \textbf{e}_{2 \leftarrow 3}, \textbf{e}_{3 \leftarrow 1}, \textbf{e}_{3 \leftarrow 2}$
are defined similarly.
\end{defn}

\begin{lem} \label{lem:avoid}
Let $A,B,C$ be three projective cameras with distinct centers.  Let 
$\pi \in (\PP^2 \sqcup (\PP^2)^{\vee})^{\times 3}$.
Assume this point/line correspondence \textup{\textbf{avoids epipoles}}.  For example,
if $\pi = (x, x', \ell'') \in \PP^{2}_{A} \times \PP^{2}_{B} \times (\PP^{2}_{C})^{\vee}$, avoidance of 
epipoles means that 
$x \neq \textbf{e}_{1\leftarrow 2}, \textbf{e}_{1\leftarrow 3};$ $x' \neq \textbf{e}_{2\leftarrow 1}, \textbf{e}_{2\leftarrow 3};$ and 
$\ell'' \niton \textbf{e}_{3\leftarrow 1}, \textbf{e}_{3\leftarrow 2}$.  Then $\pi$ is consistent with $(A,B,C)$ if $\pi$ is in the suitable multi-view variety.
\end{lem}

\begin{proof}
Assuming that $\pi$ is in the multi-view variety,
then $\pi$ satisfies the equations from Theorem \ref{thm:multi-view}.
This is equivalent to containment conditions on the \textit{back-projections} of $\pi$,
without any hypothesis on the centers of $A,B,C$.  

We spell this out for the `\textit{PPL}' case, where $\pi = (x,x',\ell'') \in \PP_{A}^{2} \times \PP_{B}^{2} \times (\PP_{C}^{2})^{\vee}$.
Here the back-projections are the lines $\alpha^{-1}(x), \beta^{-1}(x') \subseteq \PP^3$ and the plane
$\gamma^{-1}(\ell'') \subseteq \PP^3$.  The minors from Theorem \ref{thm:multi-view} vanish
if and only if there exists $(X, L) \in \mathcal{F\ell}_{0,1}$ such that $X \in \alpha^{-1}(x)$, 
$X \in \beta^{-1}(x')$ and $L \subseteq \gamma^{-1}(\ell'')$.
To see this, note that the minors vanish only if:
$$\begin{pmatrix} A & x & 0 \\ B & 0 & x' \\ \ell''^{T}C & 0 & 0 \end{pmatrix} \begin{pmatrix} X \\ -\lambda \\ -\lambda' \end{pmatrix} = 0 \,\, \textup{ for some nonzero } \,\,
\begin{pmatrix} X \\ -\lambda \\ -\lambda' \end{pmatrix} \in \CC^{6},$$
where $X \in \CC^{4}, \lambda \in \CC$ and $\lambda' \in \CC$.  
Since $x, x' \in \CC^3$ are nonzero, it follows that $X$ is nonzero, and so defines a point $X \in \PP^3$.  From $AX = \lambda x$, 
the line $\alpha^{-1}(x) \subseteq \PP^3$ contains $X \in \PP^3$.
Similarly $AX = \lambda' x$ implies $X \in \beta^{-1}(x')$.
Thirdly, $\ell''^{T}C X = 0$ says that $X$ lies on
the plane $\gamma^{-1}(\ell'') \subseteq \PP^3$.  Now taking any line $L \subseteq \PP^3$ 
with $X \in L \subseteq \gamma^{-1}(\ell'')$ produces a satisfactory point
$(X, L) \in \mathcal{F\ell}_{0,1}$, and reversing the argument gives the converse.

Returning to the lemma, since $\pi$ avoids epipoles, the back-projections
of $\pi$ avoid the centers of $A,B,C$.
In the `\textit{PPL}' case, this implies that $(X,L)$ avoids
the centers of $A,B,C$.  
Thus $(X,L)$ witnesses consistency, because $\alpha(X) = x, \beta(X) = x', \gamma(L) = \ell''$.
The other cases are finished similarly.
\end{proof}

The results of this subsection have provided
tight equational formulations for a camera configuration
and a point/line image correspondence to be consistent.
This leads to a parametrized system of polynomial equations 
for each minimal problem in Theorem \ref{thm:main}. 
For instance, for the minimal problem `$1 PPP + 4 \textit{PPL}$',
the unknowns are the entries of $A,B,C$, 
up to the action of the group $\mathcal{G}$.
Due to Theorem \ref{thm:multi-view}, there are
$\binom{9}{7} + 3 + 4 \cdot \binom{7}{6} = 67$
quartic equations. Their coefficients are
parametrized cubically and quadratically
by the image data in
$(\PP^{2})^{11} \times \big{(}(\PP^{2})^{\vee}\big{)}^{4}$.
Since this parameter space is irreducible, to find
the generic number of solutions to the system,
we may specialize to \textit{one} random instance, such as in Example
\ref{exmp:160}.
Nonetheless, solving a single instance of this system -- 
`as is' -- is computationally intractable, let alone solving 
systems for the other minimal problems present in Theorem \ref{thm:main}.

The way out is to \underline{nontrivially replace} the above systems with other systems, which enlarge
the solution sets but amount to accessible computations.  This key
maneuver is based on \textit{trifocal tensors} from multi-view geometry.
Before doing so, we justify calling the problems in Theorem \ref{thm:main} minimal.

\begin{prop} \label{prop:finite}
For each problem in Theorem \textup{\ref{thm:main}}, given generic correspondence data,
there is a finite number\footnote{\label{footnote:number} This number is shown to be positive in the proof of Theorem \ref{thm:main}.} of solutions, i.e. calibrated camera configurations $(A,B,C)$.  Moreover, solutions
have linearly independent centers.
\end{prop}

\begin{proof}
For calibrated $A,B,C$,
we may act by $\mathcal{G}$ so
$A = \begin{bmatrix} I_{3 \times 3} & 0 \end{bmatrix}$, 
$B = \begin{bmatrix} R_{2} & t_{2} \end{bmatrix}$ and $C = \begin{bmatrix} R_{3} & t_{3} \end{bmatrix}$
where $R_{2}, R_{3} \in \textup{SO}(3, \CC)$ and $t_{2}, t_{3} \in \CC^{3}$.  Furthermore,  
$t_{2}$ and $t_{3}$ may be jointly scaled.  Thus, if $A,B,C$ have non-identical centers, we get a point 
in $\textup{SO}(3, \CC)^{\times 2} \times \PP^{5}$.  This point is unique and configurations
with non-identical centers are in bijection with $\textup{SO}(3, \CC)^{\times 2} \times \PP^{5}$.

Now consider one of the minimal problems from Theorem \ref{thm:main},
`$w_{1} \textit{PPP} + w_{2} \textit{PPL} + w_{3} \textit{PLP} + w_{4} \textit{LLL} + w_{5} \textit{PLL}$'.
Notice that the problems in Theorem  \ref{thm:main}, are those for which the weights $(w_{1}, w_{2}, w_{3}, w_{4}, w_{5}) \in \ZZ_{\geq 0}$ satisfy
$3w_{1} + 2w_{2} + 2w_{3} + 2w_{4} + w_{5} = 11$ and $w_{2} \geq w_{3}$.  Image correspondence data is a point in the product
$\mathcal{D}_{w} := (\PP^{2} \times \PP^{2} \times \PP^{2})^{\times w_{1}} \times \ldots \times (\PP^{2} \times (\PP^{2})^{\vee} \times (\PP^{2})^{\vee})^{\times w_{5}}$.

Consider the incidence diagram:

\vspace{-1.1em}

$$\textup{SO}(3, \CC)^{\times 2} \times \PP^{5} \longleftarrow \Gamma \longrightarrow \mathcal{D}_{w}$$

\vspace{0.2em}

\noindent where \begin{footnotesize}$\Gamma := \overline{\{\big{(}(A,B,C), d\big{)} \in \big{(}\textup{SO}(3, \CC)^{\times 2} \times \PP^{5} \big{)} \times \mathcal{D}_{w} \,\, | \,\, (A,B,C) \textup{ and } d \textup{ are consistent} \}}$\end{footnotesize}

\vspace{0.3em}

\noindent and where the arrows are projections.  The left map is surjective and a general fiber 
is a product of multi-view varieties described by Theorem \ref{thm:multi-view}.  In particular, the fiber has
dimension $3w_{1} + 4w_{2} + 4w_{3} + 4w_{4} + 5w_{5}$.  Therefore, by \cite[Corollary 13.5]{Eis},
$\Gamma$ has dimension $11 + 3w_{1} + 4w_{2} + 4w_{3} + 4w_{4} + 5w_{5}$,
as $\dim(\textup{SO}(3, \CC)^{\times 2} \times \PP^{5}) = 11$.
Now, the second arrow is a regular map between varieties of the same dimension, because
$ 11 + 3w_{1} + 4w_{2} + 4w_{3} + 4w_{4} + 5w_{5} = 6(w_{1} + w_{2} + w_{3} + w_{4} + w_{5}) $.
So, if it is dominant, then again by \cite[Corollary 13.5]{Eis}, a general fiber has dimension 0; otherwise, a general fiber is empty.  However, note that points in a
general fiber of the second map correspond to solutions of a generic instance of the problem indexed by $w$ 
from Theorem \ref{thm:main}.  This shows that those problems generically have finitely many solutions.

We can see that generically there are no solutions with non-identical but collinear centers, as follows.  Let 
$\mathcal{C} \subset \textup{SO}(3, \CC)^{\times 2} \times \PP^{5}$ be the closed variety of configurations $(A,B,C)$
with non-identical but collinear centers.  Consider:

\vspace{-1.1em}

$$\mathcal{C} \longleftarrow \Gamma' \longrightarrow \mathcal{D}_{w}$$

\noindent where the definition of $\Gamma'$ is the definition of $\Gamma$ with 
$\textup{SO}(3, \CC)^{\times 2} \times \PP^{5}$ replaced by $\mathcal{C}$, 
and where the arrows are projections.  Here $\dim(\mathcal{C}) = 10$.   The left arrow is surjective, and
a general fiber is a product of multi-view varieties, with the same dimension as in the above case.  This dimension statement is seen by calculating
the multi-view varieties as in the proof
of Theorem \ref{thm:multi-view}, when $(A,B,C)$ have distinct, collinear centers.  It follows that
$\dim(\Gamma') = 10 + 3w_{1} + 4w_{2} + 4w_{3} + 4w_{4} + 5w_{5} < 11 + 3w_{1} + 4w_{2} + 4w_{3} + 4w_{4} + 5w_{5} = 6(w_{1} + w_{2} + w_{3} + w_{4} + w_{5}) = \dim(\mathcal{D}_{w})$ so that the right arrow is not dominant.

Finally, to see that generically there is no solution $(A,B,C)$ where the centers of $A, B, C$ are identical in $\PP^3$, 
we may mimic the above argument with another dimension count.  Calibrated configurations with identical centers 
are in bijection with $\textup{SO}(3, \CC)^{\times 2}$, because each $\mathcal{G}$-orbit has a unique representative
of the form $A = \begin{bmatrix} I_{3 \times 3} & 0 \end{bmatrix}$, $B = \begin{bmatrix} R_{2} & 0 \end{bmatrix}$, $C = \begin{bmatrix} R_{3} & 0 \end{bmatrix}$
where $R_{2}, R_{3} \in \textup{SO}(3, \CC)$.  So, analogously to before, we consider the diagram:

\vspace{-1.1em}

$$\textup{SO}(3, \CC)^{\times 2}  \longleftarrow \Gamma'' \longrightarrow \mathcal{D}_{w}$$

\noindent where the definition of $\Gamma''$ is the definition of $\Gamma$ with 
$\textup{SO}(3, \CC)^{\times 2} \times \PP^{5}$ replaced by $\textup{SO}(3, \CC)^{\times 2}$, 
and where the arrows are projections.  Again, the left arrow is surjective, and a general fiber
is a product of multi-view varieties.  Here, when $A, B, C$ have identical centers, 
a calculation as in the proof of Theorem \ref{thm:multi-view}
verifies that the dimensions of the multi-view varieties drop, as follows: 
$\dim(X_{A,B,C}^{PLL})=3, \dim(X_{A,B,C}^{LLL})= 2, \dim(X_{A,B,C}^{PPL})=3,
\dim(X_{A,B,C}^{PLP})= 3, \dim(X_{A,B,C}^{PPP})=2$.  So the dimension of a general
fiber of the left arrow is $2w_{1}+3w_{2}+3w_{3}+2w_{4}+5w_{3}$.  So
$\dim(\Gamma'') = 6 + 2w_{1}+3w_{2}+3w_{3}+2w_{4}+5w_{3} < 11 + 3w_{1} + 4w_{2} + 4w_{3} + 4w_{4} + 5w_{5} = 6(w_{1} + w_{2} + w_{3} + w_{4} + w_{5}) = \dim(\mathcal{D}_{w})$, 
whence the right arrow
is not dominant.  This completes the proof.
\end{proof}

\subsection{Trifocal tensors} \label{subsec:3.2}
In this subsection, we re-derive the trifocal tensor $T_{A,B,C} \in \CC^{3 \times 3 \times 3}$ associated to cameras $(A,B,C)$,
following the projective geometry approach of Hartley \cite{Har}.
This explains the notation in the `\textit{PLL}' bullet of Theorem \ref{thm:multi-view}, and justifies the assertion made there.
We also review how $T_{A,B,C}$ encodes other point/line images correspondences.

As in Subsection \ref{subsec:3.1}, let $A,B,C \in \CC^{3 \times 4}$ be three projective cameras, not necessarily calibrated,
and denote by $\alpha: \PP^3 \dashrightarrow \PP^{2}_{A}$, $\beta: \PP^3 \dashrightarrow \PP^{2}_{B}$,
$\gamma: \PP^3 \dashrightarrow \PP^{2}_{C}$ the corresponding linear projections.  
Let the point and lines $x \in \PP^{2}_{A}, \ell' \in (\PP^{2}_{B})^{\vee}, \ell'' \in (\PP^{2}_{C})^{\vee}$ be
given as column vectors. The pre-image $\alpha^{-1}(x)$ is a line in $\PP^3$, while
$\beta^{-1}(\ell')$ and $\gamma^{-1}(\ell'')$ are planes in $\PP^3$.  We can
characterize when these three have non-empty intersection as follows.

First, note that the plane $\beta^{-1}(\ell')$ is given by the column vector $B^{T}\ell'$
since $X \in \PP^3$ satisfies $X \in \beta^{-1}(\ell')$ if and only if $0 = \ell'^{T}BX = (B^{T}\ell')^{T}X$.
Similarly, the plane $\gamma^{-1}(\ell'')$ is given by $C^{T}\ell''$.  For the line $\alpha^{-1}(x)$, note:

\vspace{-0.4em}

$$\alpha^{-1}(x) \,\,= \bigcap_{\substack{\,\, \ell \in (\PP^{2}_{A})^{\vee} \\ \ell^{T}x = 0}} \alpha^{-1}(\ell) \,\, \subset \,\, \alpha^{-1}\langle x, \, \begin{bmatrix} 1 & 1 & 0 \end{bmatrix}^{T} \rangle  \,\,\, \cap \,\,\, \alpha^{-1}\langle x, \, \begin{bmatrix} 1 & 0 & 1 \end{bmatrix}^{T} \rangle. $$

\vspace{-0.4em}

\noindent Here $\langle \,\, \rangle$ denotes span, and auxiliary points $\begin{bmatrix} 1 & 1 & 0 \end{bmatrix}^{T}, \begin{bmatrix} 1 & 0 & 1 \end{bmatrix}^{T} \in \PP^{2}_{A}$
are simply convenient choices for this calculation.  Unless those two points and $x$ are collinear, the inclusion above is an equality, and the intersectands
in the RHS are the planes given by the column vectors $A^{T} \, [x]_{\times} \begin{bmatrix} 1 & 1 & 0 \end{bmatrix}^{T}$ and $A^{T} \, [x]_{\times} \begin{bmatrix} 1 & 0 & 1 \end{bmatrix}^{T}$.  The~notation~means~$[x]_{\times}=\begin{bmatrix} 0 & -x_{3} & x_{2} \\ x_{3} & 0 & -x_{1} \\ -x_{2} & x_{1} & 0 \end{bmatrix}$,~and $[x]_{\times}y$ gives $\langle x, y  \rangle$ for $x\neq y \in \PP^{2}_{A}$.  So, $\alpha^{-1}(x) \cap \beta^{-1}(\ell') \cap \gamma^{-1}(\ell'') \neq \emptyset$ only if:

\vspace{-0.1em}

\begin{equation} \label{eqn:trifocaldet}
\textup{det}\begin{pmatrix} A^{T} \, [x]_{\times} \begin{bmatrix} 1 \\ 1 \\ 0 \end{bmatrix} \,\, \Big | \,\, A^{T} \, [x]_{\times} \begin{bmatrix} 1 \\ 0 \\ 1 \end{bmatrix} \,\, \Big | \,\, B^{T}\ell' \,\, \Big | \,\, B^{T}\ell'' \end{pmatrix}_{4 \times 4} \ = \,\,\,\,\, 0.
\end{equation}

\vspace{0.25em}

\noindent This determinant is divisible by $(x_{1} - x_{2} - x_{3})$, since that vanishes if and only if
$x, \begin{bmatrix} 1 & 1 & 0 \end{bmatrix}^{T}, \begin{bmatrix} 1 & 0 & 1 \end{bmatrix}^{T}$ are collinear only if
the first two columns above are linearly dependent.  Hence, factoring out, we obtain a constraint
that is trilinear in $x, \ell', \ell''$, i.e., we get for some tensor $T \in \CC^{3 \times 3 \times 3}$:
$$ \sum_{1 \leq i,j,k \leq 3} T_{ijk} \, x_{i}\,\ell'_{j}\,\ell''_{k} = 0.$$

\noindent The tensor entry $T_{ijk}$ is computed by substituting into (\ref{eqn:trifocaldet}) the basis vectors 
$x = \textbf{e}_{i}, \ell' = \textbf{e}_{j}, \ell'' = \textbf{e}_{k}$.  Breaking into cases according to $i$, this yields:

\vspace{0.3cm}

\begin{itemize}
\setlength\itemsep{0.7em}
\item $T_{1ij} = {\scriptstyle \frac{1}{(1 - 0 - 0)}} \,\, \textup{det}\begin{pmatrix} \textbf{a}_{3} \, \big | -\textbf{a}_{2} \, \big | \, \textbf{b}_{j} \, \big | \, \textbf{c}_{k} \end{pmatrix} 
= \textup{det}\begin{pmatrix} \textbf{a}_{2} \, \big | \, \textbf{a}_{3} \, \big | \, \textbf{b}_{j} \, \big | \, \textbf{c}_{k} \end{pmatrix}$
\item $T_{2ij} = {\scriptstyle \frac{1}{(0 - 1 - 0)}} \,\, \textup{det}\begin{pmatrix} -\textbf{a}_{3} \, \big | \textbf{a}_{1} -\textbf{a}_{3} \, \big | \, \textbf{b}_{j} \, \big | \, \textbf{c}_{k} \end{pmatrix} 
= -\textup{det}\begin{pmatrix} \textbf{a}_{1} \, \big | \, \textbf{a}_{3} \, \big | \, \textbf{b}_{j} \, \big | \, \textbf{c}_{k} \end{pmatrix}$
\item $T_{3ij} = {\scriptstyle \frac{1}{(0 - 0 - 1)}} \,\, \textup{det}\begin{pmatrix} -\textbf{a}_{1}+\textbf{a}_{2} \, \big | \textbf{a}_{2} \, \big | \, \textbf{b}_{j} \, \big | \, \textbf{c}_{k} \end{pmatrix} 
= \textup{det}\begin{pmatrix} \textbf{a}_{1} \, \big | \, \textbf{a}_{2} \, \big | \, \textbf{b}_{j} \, \big | \, \textbf{c}_{k} \end{pmatrix}$
\end{itemize}

\vspace{0.2cm}

\noindent where $\textbf{a}_{i}$ denotes the transpose of the first row in $A$, and so on. 

At this point, we have derived formula (17.12) from \cite[pg 415]{HZ}:

\begin{defn} \label{defn:main}
Let $A,B,C$ be cameras.
Their \textup{\textbf{trifocal tensor}} $T_{A,B,C} \in \CC^{3 \times 3 \times 3}$
is computed as follows.  Form the $4\times 9$ matrix
$\begin{pmatrix} A^{T} \big | B^{T} \big | C^{T}  \end{pmatrix}$.
Then for $1 \leq i,j,k \leq 3$, the entry $(T_{A,B,C})_{ijk}$ is $(-1)^{i+1}$ times the determinant
of the $4 \times 4$ submatrix gotten by \textup{omitting} the $i^{\textup{th}}$ column from $A^{T}$,
while \textup{keeping} the $j^{\textup{th}}$ and $k^{\textup{th}}$ columns from $B^{T}$ and $C^{T}$, respectively.
If $A,B,C$ are calibrated, then $T_{A,B,C}$ is said to be a \textup{\textbf{calibrated trifocal tensor}}.
\end{defn}

\begin{rem}
Since $A, B, C \in \CC^{3 \times 4}$ are each defined only up to multiplication by a nonzero scalar,
the same is true of $T_{A,B,C} \in \CC^{3 \times 3 \times 3}$.
\end{rem}

\begin{rem}
By construction, $T_{A,B,C}(x, \ell', \ell'') := \sum_{1 \leq i,j,k \leq 3} T_{ijk} \, x_{i}\,\ell'_{j}\,\ell''_{k} = 0$ is equivalent 
to $\alpha^{-1}(x) \cap \beta^{-1}(\ell') \cap \gamma^{-1}(\ell'') \neq \emptyset$.  In particular, $T_{A,B,C} = 0$ if
and only if the centers of $A$, $B$, $C$ are all the same.  Moreover, the `\textit{PLL}' cases in 
Theorem \ref{thm:multi-view} and Lemma \ref{lem:avoid} postponed above are now immediate.
\end{rem}

So far, we have constructed trifocal tensors so that they encode point-line-line image correspondences.
Conveniently, the same tensors encode other point/line correspondences \cite{Har}, up to extraneous components.

\begin{prop}\label{prop:lineartrifocal}
Let $A,B,C$ be projective cameras.
Let $x \in \PP^{2}_{A}, x' \in \PP^{2}_{B}, x'' \in \PP^{2}_{C}$ and $\ell \in (\PP^{2}_{A})^{\vee}, \ell' \in (\PP^{2}_{B})^{\vee}, \ell'' \in (\PP^{2}_{C})^{\vee}$.
Putting $T = T_{A,B,C}$,
then $(A,B,C)$ is consistent with\textup{:}

\begin{itemize}
\setlength\itemsep{0.6em}
\item{\makebox[7.25cm]{$(x, \ell', \ell'')$ only if $T(x, \ell', \ell'') = 0$\hfill} \textup{\textbf{[PLL]}}}
\item{\makebox[7.25cm]{$(\ell, \ell', \ell'')$ only if $[\ell]_{\times}T( \--, \ell', \ell'') = 0$\hfill} \textup{\textbf{[LLL]}}}
\item{\makebox[7.25cm]{$(x, \ell', x'')$ only if $[x'']_{\times}T( x, \ell', \--) = 0$\hfill} \textup{\textbf{[PLP]}}}
\item{\makebox[7.25cm]{$(x, x', \ell'')$ only if  $[x']_{\times}T( x, \--, \ell'') = 0$\hfill} \textup{\textbf{[PPL]}}}
\item{\makebox[7.25cm]{$(x, x', x'')$ only if $[x'']_{\times}T(x, \--, \--)[x']_{\times}=0$.\hfill} \textup{\textbf{[PPP]}}}
\end{itemize}
\vspace{0.3em}
In the middle bullets, each contraction of $T$ with two vectors gives a column vector in $\CC^3$.  In the last bullet,
$T(x, \--, \--) = \sum_{i=1}^{3} x_{i}(T_{ijk})_{1\leq j,k \leq 3} \in \CC^{3\times 3}$.
\end{prop}

\begin{proof}
This proposition matches Table 15.1 on \cite[pg 372]{HZ}.  To be self-contained, we recall the proof.
The first bullet is by construction of $T$. 

For the second bullet, 
assume that $(\ell, \ell', \ell'')$ is consistent with $(A,B,C)$, i.e. there exists $L \in \textup{Gr}(\PP^{1}, \PP^{3})$
such that $\alpha(L) = \ell, \beta(L) = \ell', \gamma(L) = \ell''$.  Now let $y \in \ell$ be a point.
So $\alpha^{-1}(x)$ is a line in the plane $\alpha^{-1}(\ell)$ and that plane contains the line $L$.
This implies $\alpha^{-1}(x) \cap L \neq \emptyset \Rightarrow \alpha^{-1}(x) \cap \beta^{-1}(\ell') \cap \gamma^{-1}(\ell'') \neq \emptyset \Leftrightarrow T(y, \ell', \ell'') = 0$.
It follows that for $y \in \PP^{2}_{A}$, we have 
$y^{T} \ell = 0 \Rightarrow y^{T} T(\--, \ell', \ell'') = 0$. This means that $\ell$ and $T(\--, \ell', \ell'')$ 
are linearly independent, i.e. $[\ell]_{\times}T( \--, \ell', \ell'') = 0$.

The third, fourth and fifth bullets are similar.  They come from reasoning that
the consistency implies, respectively:
\vspace{0.1em}
\begin{itemize}
\setlength\itemsep{0.4em}
\item $x'' \in \kay'' \Rightarrow T(x, \ell', \kay'') = 0$
\item $x' \in \kay' \Rightarrow T(x, \kay', \ell'') = 0$
\item ${\big{(}}x' \in \kay'$ and $x'' \in \kay''{\big{)}} \Rightarrow T(x, \kay', \kay'') =0$,
\end{itemize}
\vspace{0.2em}
where $\kay' \in (\PP^{2}_{B})^{\vee}$ and $\kay'' \in (\PP^{2}_{C})^{\vee}$.
\end{proof}

\begin{rem}
The constraints in Proposition \ref{prop:lineartrifocal} are linear in $T$.  We will exploit this in Section \ref{sec:6}.
Also, in fact, image correspondences of types `\textit{LPL}', `\textit{LLP}' and `\textit{LPP}' do \textit{not} give linear constraints on $T_{A,B,C}$.
This is the reason that these types are not considered in Theorem \ref{thm:main}.
To get linear constraints nonetheless, one could permute $A,B,C$ before forming the trifocal tensor.
\end{rem}

In this subsection, we have presented a streamlined account of trifocal tensors, and 
the point/line image correspondences that they encode.  Now, we sketch the relationship
between the \textit{tight} conditions in Theorem \ref{thm:multi-view} and the \textit{necessary} conditions in Proposition \ref{prop:lineartrifocal}
for consistency.

\begin{lem}\label{lem:topdim}
Fix projective cameras $A,B,C$ with linearly independent centers.  Then the trilinearities in Proposition \ref{prop:lineartrifocal} cut out subschemes of three-factor products of $\PP^{2}$ and $(\PP^{2})^{\vee}$.  In all cases of Proposition \ref{prop:lineartrifocal}, this subscheme is reduced and contains the corresponding multi-view variety
as a top-dimensional component.
\end{lem}

\begin{proof}
Without loss of generality, $A, B, C$ are
in the special position from the proof of Theorem \ref{thm:multi-view}.
Then using \texttt{Macaulay2}, we form the ideal 
generated by the trilinearities of Proposition \ref{prop:lineartrifocal} and saturate with respect
to the irrelevant ideal.  This leaves a radical ideal; we compute its primary decomposition.  
\end{proof}

For example, in the case of `\textit{PPP}', the trilinearities from Proposition \ref{prop:lineartrifocal}
generate a radical ideal in $\CC[x_i, x'_j, x''_k]$ that is the intersection of:
\begin{itemize}
\setlength\itemsep{0.15em}
\item the 3 irrelevant ideals for each factor of $\PP^{2}$
\item 2 linear ideals of codimension $4$
\item the multi-view ideal $I(X^{PPP}_{A,B,C})$.
\end{itemize}
This discrepancy between the trifocal and multi-view conditions 
for `\textit{PPP}' correspondences was studied in \cite{TPH}.
To demonstrate our main result, in Section \ref{sec:6} we shall relax the tight multi-view equations in Theorem \ref{thm:multi-view}
to the merely necessary trilinearities in Proposition \ref{prop:lineartrifocal}.  The `top-dimensional' clause in Lemma \ref{lem:topdim}, as well as Theorem \ref{thm:calibratedConfig} in Section \ref{sec:4} below, indicate that
this gives `good' approximations to the minimal problems in Theorem \ref{thm:main}.

\section{Configurations}\label{sec:4} In this section, it is proven
that trifocal tensors, in both the uncalibrated and calibrated case, 
are in bijection with camera triples up to the appropriate group 
action, i.e. with camera configurations.  Statements tantamount to Proposition
\ref{prop:uncalibratedConfig} are made throughout \cite[Chapter 15]{HZ} and are
well-known in the vision community,
however, we could not find any proof in the literature.
As far as our main result Theorem \ref{thm:main} is concerned,
Theorem \ref{thm:calibratedConfig} below enables
us to compute consistent calibrated trifocal tensors in exchange for
consistent calibrated camera configurations.  To our knowledge, this
theorem is new; subtly, the analog for two calibrated is false \cite[Result 9.19]{HZ}.

\begin{prop} \label{prop:uncalibratedConfig}
Let $A,B,C$ be three projective cameras, with linearly independent centers in $\PP^{3}$
Let $\widetilde{A},\widetilde{B},\widetilde{C}$ be another three projective cameras.
Then $T_{A,B,C} = T_{\widetilde{A},\,\widetilde{B},\,\widetilde{C}} \in \PP(\CC^{3 \times 3 \times 3})$ if and only if 
there exists $h \in \textup{SL}(4, \CC)$ such that $Ah = \widetilde{A}, \, Bh = \widetilde{B}, \, Ch = \widetilde{C} \, \in \, \PP(\CC^{3 \times 4})$.
\end{prop}

\begin{proof}
As in the proof of Theorem \ref{thm:multi-view}, for $g, g', g'' \in \textup{SL}(3, \CC), \, h \in \textup{SL}(4, \CC)$:

\vspace{-0.8em}

\begin{equation}\label{eqn:group}
T_{gA, \, g'B, \, g''C} = (g, \wedge^{2} g', \wedge^{2} g'') \cdot T_{A,B,C} \,\,\, \textup{ and } \,\,\, T_{Ah, \, Bh, \, Ch} = T_{A, B, C}.
\end{equation}

\vspace{0.1em}

\noindent The second equality gives the `if' direction.  Conversely, for `only if',  
for any $g, g', g'' \in \textup{SL}(3, \CC)$, $h_{1}, h_{2} \in \textup{SL}(4, \CC)$, we are free to replace $(A,B,C)$ by $(gAh_{1}, g'Bh_{1}, g''Ch_{1})$ and to replace
$(\widetilde{A}, \widetilde{B}, \widetilde{C})$ by $(g\widetilde{A}h_{2}, g'\widetilde{B}h_{2}, g''\widetilde{C}h_{2})$, and then to
exhibit an $h$ as in the proposition.  Hence we may assume that:

$$A = \begin{bmatrix} 1 & 0 & 0 & 0 \\ 0 & 1 & 0 & 0 \\ 0 & 0 & 1 & 0 \end{bmatrix}\!, \,\,
B = \begin{bmatrix} 1 & 0 & 0 & 0 \\ 0 & 1 & 0 & 0 \\ 0 & 0 & 0 & 1 \end{bmatrix}\!, \,\,
C = \begin{bmatrix} 1 & 0 & 0 & 0 \\ 0 & 0 & 1 & 0 \\ 0 & 0 & 0 & 1 \end{bmatrix}$$

\vspace{0.4em}

$$\widetilde{A} = \begin{bmatrix} 1 & 0 & 0 & 0 \\ 0 & 1 & 0 & 0 \\ 0 & 0 & 1 & 0 \end{bmatrix}\!, \,\,
\widetilde{B} = \begin{bmatrix} \ast & \ast & \ast & \ast \\ \ast & \ast & \ast & \ast \\ 0 & 0 & 0 & 1 \end{bmatrix}\!, \,\,
\widetilde{C} = \begin{bmatrix} \ast & \ast & \ast & \ast \\ \ast & \ast & \ast & \ast \\ \ast & \ast & \ast & \ast \end{bmatrix}$$

\vspace{0.6em}

\noindent where each `$\ast$' denotes an indeterminate.  Now consider the nine equations:

\vspace{-0.5em}

$$(T_{A,B,C})_{i\,3\,k} = (T_{\widetilde{A},\, \widetilde{B},\, \widetilde{C}})_{i\,3\,k}$$

\vspace{0.4em}

\noindent where $1 \leq i, k \leq 3$.  Under the above assumptions, these are linear and in the nine unknowns $\widetilde{c}_{lm}$
for $1 \leq l, m \leq 3$.  Here we have fixed the nonzero scale on $\widetilde{C}$ so that these are indeed equalities,
on the nose.  It follows that:

\vspace{-0.4em}

$$\widetilde{C} = \begin{bmatrix} 1 & 0 & 0 & \ast \\ 0 & 0 & 1 & \ast \\ 0 & 0 & 0 & \ast \end{bmatrix}.$$

\vspace{0.2em}

\noindent At this point, we have reduced to solving $18$ equations in 11 unknowns:

\vspace{-0.7em}

$$(T_{A,B,C})_{i\,j\,k} = (T_{\widetilde{A},\, \widetilde{B},\, \widetilde{C}})_{i\,j\,k}$$

\vspace{0.2em}

\noindent where $1 \leq i, k \leq 3$ and $1 \leq j \leq 2$.  These equations are quadratic monomials
and binomials.  The system is simple to solve by hand or with \texttt{Macaulay2}:

\vspace{-0.4em}

$$\widetilde{A} = \begin{bmatrix} 1 & 0 & 0 & 0 \\ 0 & 1 & 0 & 0 \\ 0 & 0 & 1 & 0 \end{bmatrix}\!, \,\,
\widetilde{B} = \begin{bmatrix} \lambda & 0 & 0 & 0 \\ 0 & \lambda & 0 & 0 \\ 0 & 0 & 0 & 1 \end{bmatrix}\!, \,\,
\widetilde{C} = \begin{bmatrix} 1 & 0 & 0 & 0 \\ 0 & 0 & 1 & 0 \\ 0 & 0 & 0 & \lambda^{-1} \end{bmatrix}$$

\vspace{0.2em}

\noindent for $\lambda \in \CC^{*}$.  Now taking $h = \lambda^{-3/4} \, \textup{diag}(\lambda, \lambda, \lambda, 1) \in \textup{SL}(4, \CC)$
gives $Ah = \widetilde{A}, Bh = \widetilde{B}, Ch = \widetilde{C} \in \PP(\CC^{3 \times 4})$, as desired.  This completes the proof.
\end{proof}

With a bit of work, we can promote Proposition \ref{prop:uncalibratedConfig} to the calibrated case.

\begin{thm} \label{thm:calibratedConfig} 

Let $A,B,C$ be three calibrated cameras, with linearly independent centers in $\PP^{3}$.
Let $\widetilde{A},\widetilde{B},\widetilde{C}$ be another three calibrated cameras.
Then $T_{A,B,C} = T_{\widetilde{A},\,\widetilde{B},\,\widetilde{C}} \in \PP(\CC^{3 \times 3 \times 3})$ if and only if 
there exists $h \in \mathcal{G}$ (where $\mathcal{G}$ is defined on page \textup{2}) such that $Ah = \widetilde{A}, \, Bh = \widetilde{B}, \, Ch = \widetilde{C} \, \in \, \PP(\CC^{3 \times 4})$.

\end{thm}

\begin{proof}
The `if' direction is from Proposition \ref{prop:uncalibratedConfig}.  For `only if', here
for any $g, g', g'' \in \textup{SO}(3, \CC)$, $h_{1}, h_{2} \in \mathcal{G}$, we are free to replace $(A,B,C)$ by $(gAh_{1}, g'Bh_{1}, g''Ch_{1})$ and to replace
$(\widetilde{A}, \widetilde{B}, \widetilde{C})$ by $(g\widetilde{A}h_{2}, g'\widetilde{B}h_{2}, g''\widetilde{C}h_{2})$, and then to
exhibit an $h \in \mathcal{G}$ as above.  In this way, we may assume that:

\vspace{-0.5em}

$$A = \begin{bmatrix} I_{3 \times 3} & 0  \end{bmatrix}\!, \,\,
B = \begin{bmatrix} I_{3 \times 3} & s_{1} \end{bmatrix}\!, \,\, 
C = \begin{bmatrix} I_{3 \times 3} & s_{2} \end{bmatrix} $$

\vspace{0.05em}

$$\widetilde{A} = \begin{bmatrix} I_{3 \times 3} & 0  \end{bmatrix}\!, \,\,
\widetilde{B} = \begin{bmatrix} R_{1} & t_{1} \end{bmatrix}\!, \,\, 
\widetilde{C} = \begin{bmatrix} R_{2} & t_{2} \end{bmatrix} $$

\vspace{0.4em}

\noindent where $R_{1}, R_{2} \in \textup{SO}(3, \CC)$ and $s_{1}, s_{2}, t_{1}, t_{2} \in \CC^{3}$.
Now from Proposition \ref{prop:uncalibratedConfig}, there exists $h' \in \textup{SL}(4, \CC)$ such 
that $Ah' = \widetilde{A}, \, Bh' = \widetilde{B}, \, Ch' = \widetilde{C} \in \PP(\CC^{3 \times 3})$.
From the first equality, it follows that $h' = \begin{bmatrix} I_{3\times 3} & 0 \\[1pt] u^{T} & \lambda \end{bmatrix} \in \PP(\CC^{4 \times 4})$
for some $u \in \CC^{3}, \, \lambda \in \CC^{*}$.  
It suffices to show that $u = 0$, so $h' \in \mathcal{G}$.
By way of contradiction, let us assume that $u \neq 0$.
Substituting into $Bh' = \widetilde{B}$ gives:

\vspace{-0.5em}

$$ \begin{bmatrix} I_{3 \times 3} & s_{1} \end{bmatrix} \, \begin{bmatrix} I_{3 \times 3} & 0 \\[1pt] u^{T} & \lambda \end{bmatrix} \,\, = \,\,
\begin{bmatrix} I_{3 \times 3}+s_{1}u^{T} & \lambda s_{1} \end{bmatrix}
\,\, = \,\, \begin{bmatrix} R_{1} & t_{1} \end{bmatrix} \, \in \PP(\CC^{3 \times 4}).
$$

\vspace{0.3em}

\noindent In particular, there is $\mu_{1} \in \CC^{*}$ so that $\mu_{1}(I_{3 \times 3} + s_{1}u^{T}) = R_{1}.$  In particular, $R_{1} - \mu_{1}I_{3 \times 3}$ is rank at most 1.
Equivalently, $\mu_{1}$ is an eigenvalue of the rotation $R_{1} \in \textup{SO}(3, \CC)$ of geometric multiplicity at least 2.
The only possibilities are $\mu_{1} = 1$ or $\mu_{1} = -1$.  If $\mu_{1} = 1$, then
$R_{1} = I$ and $s_{1} u^{T} = 0$.  From $u \neq 0$, we get that $s_{1} = 0$; but then $A = B$, contradicting linear independence of the 
centers of $A, B, C$.  So in fact $\mu_{1} = -1$.  Now $R_{1}$ is a $180\degree$ rotation.  From $R_{1} + I_{3 \times 3} = s_{1}u^{T} \in \CC^{3 \times 3}$,
it follows that the axis of rotation is the line through $u$, and $s_{1} =  \frac{2u}{u^{T}u}$. 
The exact same analysis holds starting from $Ch' = \widetilde{C}$.   So in particular, $s_{2} = \frac{2u}{u^{T}u}$.
But now $B = C$, contradicting linear independence of the 
centers of $A, B, C$.  We conclude that $u=0$. 
\end{proof}

\section{Varieties} \label{sec:5}
So far in Subsection \ref{subsec:3.2} and Section \ref{sec:4}, we have worked with individual trifocal tensors, uncalibrated or calibrated. 
This is possible once a camera configuration $(A,B,C)$ is given. 
To determine an unknown camera configuration from image data, we need to work with the set of all trifocal tensors.

\begin{defn} \label{defn:trifocal}
The \textup{\textbf{trifocal variety}}, denoted $\mathcal{T} \subset \PP(\CC^{3 \times 3 \times 3})$, is defined to be the Zariski closure
of the image of the following rational map\textup{:}

\vspace{-1.2em}
$$
\PP(\CC^{3 \times 4}) \times \PP(\CC^{3 \times 4}) \times \PP(\CC^{3 \times 4}) \dashrightarrow \PP(\CC^{3 \times 3 \times 3}), \,\,\, (A,B,C) \mapsto T_{A,B,C} $$
\vspace{-1em}
$$\textup{where } \,
(T_{A,B,C})_{ijk} := (-1)^{i+1} \textup{det} \begin{bmatrix} \sim\textup{\textbf{a}}_{i} \\ \textup{\textbf{b}}_{j} \\ \textup{\textbf{c}}_{k} \end{bmatrix}_{ 4 \times 4} \!\!\!\!\!\!  \textup{ for } 1 \leq i,j,k \leq 3.
$$

\vspace{-0.2em}

\noindent Here $\sim\textup{\textbf{a}}_{i}$ is gotten from $A$ by omitting the $i^{\textup{th}}$ row, and $\textup{\textbf{b}}_{j}, \textup{\textbf{c}}_{k}$ are the $j^{\textup{th}}, k^{\textup{th}}$ rows of $B, C$ respectively.  So, $\mathcal{T}$ is the closure of the set of all
trifocal tensors.

\end{defn}

\begin{defn} \label{defn:caltrifocal}
The \textup{\textbf{calibrated trifocal variety}}, denoted $\mathcal{T}_{\textup{cal}} \subset \PP(\CC^{3 \times 3 \times 3})$, is defined to be the Zariski closure
of the image of the following rational map\textup{:}
\vspace{-0.8em}
$$
(\textup{SO}(3, \CC) \times \CC^{3}) \, \times \, (\textup{SO}(3, \CC) \times \CC^{3}) \, \times \, (\textup{SO}(3, \CC) \times \CC^{3}) 
\, \dashrightarrow \, \PP(\CC^{3 \times 3 \times 3}), 
$$
$$
\,\,\, \Big{(}(R_{1}, t_{1}),\,(R_{2}, t_{2}),\,(R_{3}, t_{3})\Big{)} \mapsto T_{[R_{1} | t_{1}],\,[R_{2} | t_{2}],\,[R_{3} | t_{3}]}
$$
\!where the formula for $T$ is as in Definitions \textup{\ref{defn:main}} and \textup{\ref{defn:trifocal}}.  So, $\mathcal{T}_{\textup{cal}}$ is the
closure of the set of all calibrated trifocal tensors.
\end{defn}

In the remainder of this paper, the calibrated trifocal variety $\mathcal{T}_{\textup{cal}}$ is the main actor.  It has recently been studied independently by Martyushev \cite{Mar} and Matthews \cite{Mat}.  They both obtain implicit quartic equations for $\mathcal{T}_{\textup{cal}}$.
However, a full set of ideal generators for $I(\mathcal{T}_{\textup{cal}}) \subset \CC[T_{ijk}]$ is currently
not known.
We summarize the state of knowledge on implicit equations for $\mathcal{T}_{\textup{cal}}$:

\begin{prop}\label{prop:implicit}
The prime ideal of the calibrated trifocal variety $I(\mathcal{T}_{\textup{cal}}) \subset \CC[T_{ijk}]$ contains the ideal
of the trifocal variety $I(\mathcal{T})$, and $I(\mathcal{T})$ is minimally generated by \textup{10} cubics, \textup{81} quintics and
\textup{1980} sextics.  Additionally, $I(\mathcal{T}_{\textup{cal}})$ contains \textup{15} linearly independent quartics that
do not lie in $I(\mathcal{T})$.
\end{prop}

The ideal containment follows from $\mathcal{T}_{\textup{cal}} \subset \mathcal{T}$, and the statement about minimal generators of $I(\mathcal{T})$ was proven by Aholt and Oeding \cite{AO}.  For the additional quartics, see \cite[Theorems 8, 11]{Mar} and \cite[Corollary 51]{Mat}.

In the rest of this paper, using \underline{numerical algebraic geometry}, we always interact with the calibrated trifocal variety 
$\mathcal{T}_{\textup{cal}}$
directly via (a restriction of) its defining parametrization.  Therefore, we do not need the ideal of implicit equations $I(\mathcal{T}_{\textup{cal}})$, nor do we use the known equations from Proposition \ref{prop:implicit}.

At this point, we discuss properties of the rational map in Definition \ref{defn:caltrifocal}.  First, since the source $(\textup{SO}(3, \CC) \times \CC^{3})^{\times 3}$ is irreducible, the closure of the image $\mathcal{T}_{\textup{cal}}$ is irreducible.  Second, the base locus of the map 
consists of triples of calibrated cameras $\big{(}[R_{1} | t_{1}],\,[R_{2} | t_{2}],\,[R_{3} | t_{3}]\big{)}$ all with the same center in $\PP^3$, by the
remarks following Definition \ref{defn:main}.  Third, the two equations in  (\ref{eqn:group}), the second line of the proof of Proposition \ref{prop:uncalibratedConfig}, mean that the rational map
in Definition \ref{defn:caltrifocal} satisfies group symmetries.  Namely, the parametrization of $\mathcal{T}_{\textup{cal}}$ is \textit{equivariant} with respect to $\textup{SO}(3, \CC)^{\times 3}$, and each of its fibers carry a $\mathcal{G}$ 
action.  In vision, these two group actions are interpreted as changing image coordinates and changing world coordinates. 
Here, by the equivariance, it follows that $\mathcal{T}_{\textup{cal}}$ is an $\textup{SO}(3, \CC)^{\times 3}$-variety.
Also, we can use the $\mathcal{G}$ action on fibers to pick out one point per fiber, and thus restrict the map in Definition \ref{defn:caltrifocal}
so that the restriction is generically injective and dominant onto $\mathcal{T}_{\textup{cal}}$.  Explicitly, we 
restrict to the domain where $[R_{1} \, | \, t_{1}] = \begin{bmatrix} I_{3 \times 3} & 0 \end{bmatrix}\!, \,\, t_{2} = \begin{bmatrix} \ast & \ast & 1 \end{bmatrix}^{T}$.
This restriction $ (\textup{SO}(3, \CC) \times \CC^{2}) \, \times \, (\textup{SO}(3, \CC) \times \CC^{3}) 
\, \dashrightarrow \, \mathcal{T}_{\textup{cal}}$ is generically injective by Theorem \ref{thm:calibratedConfig}.
Generic injectivity makes the restricted map particularly amenable to numerical algebraic geometry, where \textit{computations regarding a parametrized
variety are pulled back to the source of the parametrization}.  We now obtain the major theorem of this section using that technique:

\begin{thm} \label{thm:4912}
The calibrated trifocal variety $\mathcal{T}_{\textup{cal}} \subset \PP(\CC^{3 \times 3 \times 3})$ is irreducible, dimension $\textup{11}$ and degree $\textup{4912}$.
It equals the $\textup{SO}(3, \CC)^{\times 3}$-orbit closure generated by the following projective plane, 
parametrized by $\begin{bmatrix} \lambda_{1} & \lambda_{2} & \lambda_{3} \end{bmatrix}^{T} \in \PP^{2}\textup{:}$ 
{\footnotesize{
$$
T_{1\ast\ast} = \begin{bmatrix} 0 & \lambda_{1} & \lambda_{2} \\ 0 & 0 & 0 \\ \lambda_{1} & 0 & 0 \end{bmatrix}\!\!, \,\,\, 
T_{2\ast\ast} = \begin{bmatrix} 0 & 0 & 0 \\ 0 & \lambda_{1} & \lambda_{2} \\ 0 & \lambda_{3} & 0 \end{bmatrix}\!\!, \,\,\,
T_{3\ast\ast} = \begin{bmatrix} 0 & 0 & 0 \\ 0 & 0 & 0 \\ 0 & \lambda_{1} & \lambda_{2} + \lambda_{3} \end{bmatrix}\!\!.
$$
}}
\end{thm}

\begin{proof}
[Computational Proof]
Dimension 11 follows from the generically injective parametrization given above.  The $\textup{SO}(3, \CC)^{\times 3}$ statement
follows from (\ref{eqn:group}).  In more detail,
given a calibrated camera configuration $(A,B,C)$ with linearly independent centers, we may act by $\mathcal{G}$ so that
the centers of $A, B, C$ are:
{\small{
$$\begin{bmatrix} 0 & 0 & 0 & 1\end{bmatrix}^{T}\!\!\!, \,\,\,\,\, \begin{bmatrix} 0 & 0 & 1 & 1\end{bmatrix}^{T}\!\!\!, \,\,\,\,\,
\begin{bmatrix} 0 & \ast & \ast & 1\end{bmatrix}^{T}\!\!\!,$$}}
\!\!respectively.  Then we may act by $\textup{SO}(3, \CC)^{\times 3}$ so that the left submatrices of $A,B,C$ equal $I_{3\times 3}$.
The calibrated trifocal tensor $T_{A,B,C}$ now lands in the stated $\PP^{2}$.   Hence, $\mathcal{T}_{\textup{cal}}$ is that orbit closure due to transformation laws (\ref{eqn:group}).

To compute the degree of $\mathcal{T}_{\textup{cal}}$, we use the open-source homotopy continuation software 
\texttt{Bertini}.  We fix a random linear subspace $\textbf{L} \subset \PP(\CC^{3 \times 3 \times 3})$ of complementary dimension to $\mathcal{T}_{\textup{cal}}$,
i.e. $\textup{dim}(\textbf{L}) = 15$. This is expressed in floating-point as the vanishing of $11$ random linear
forms \begin{inlineequation}[eqn:linearforms]{\ell_{m}(T_{ijk}) = 0}\end{inlineequation}, where $m = 1, \ldots, 11$.  Our goal is to compute $\#(\mathcal{T}_{\textup{cal}} \cap \textbf{L})$.  As homotopy continuation calculations are sensitive to the formulation used, we carefully
explain our own formulation to calculate $\mathcal{T}_{\textup{cal}} \cap \textbf{L}$.  Our formulation starts with the parametrization of 
$\mathcal{T}_{\textup{cal}}$ above, and with its two copies of $\textup{SO}(3,\CC)$.

Recall that unit norm quaternions double-cover $\textup{SO}(3, \RR)$.  Complexifying:

\vspace{-0.4em}

$$
R_{2} = \begin{pmatrix}
a^{2} + b^{2} - c^{2} - d^{2} & 2(bc - ad) & 2(bd + ac) \\
2(bc+ad) & a^{2} + c^{2} - b^{2} - d^{2} & 2(cd - ab) \\
2(bd - ac) & 2(cd + ab) & a^{2} + d^{2} - b^{2} - c^{2}
\end{pmatrix}
$$

\vspace{0.4em}

\noindent where $a,b,c,d \in \CC$ and \begin{inlineequation}[eqn:firstSO]{a^{2} + b^{2} + c^{2} + d^{2} = 1}\end{inlineequation}. Similarly for $R_{3}$ with $e,f,g,h \in \CC$ subject
to \begin{inlineequation}[eqn:secondSO]{e^{2} + f^{2} + g^{2} + h^{2} = 1}\end{inlineequation}.
For our purposes, it is computationally advantageous to replace (\ref{eqn:firstSO}) by a random patch 
\begin{inlineequation}[eqn:firstPatch]{\alpha_{1}a + \alpha_{2}b + \alpha_{3}c + \alpha_{4}d = 1}\end{inlineequation}, where 
$\alpha_{i} \in \CC$ are random floating-point numbers fixed once and for all.  
Similarly, we replace (\ref{eqn:secondSO}) by a random patch 
\begin{inlineequation}[eqn:secondPatch]{\beta_{1}e + \beta_{2}f + \beta_{3}g + \beta_{4}h = 1}\end{inlineequation}.
The patches (\ref{eqn:firstPatch}) and (\ref{eqn:secondPatch}) leave us with injective parameterizations of two subvarieties of $\CC^{3 \times 3}$, that we denote by $\textup{SO}(3,\CC)^{\boldsymbol{\alpha}}, \textup{SO}(3,\CC)^{\boldsymbol{\beta}}$.  These two varieties
have the same closed \textit{affine cone} as the closed affine cone of $\textup{SO}(3, \CC)$.  This affine cone is:

\vspace{-0.85em}

$$\widehat{\textup{SO}(3, \CC)} := \{ R \in \CC^{3 \times 3} : \exists \, \lambda \in \CC \textup{ s.t. } RR^{T} = R^{T}R = \lambda I_{3 \times 3}\}$$

\vspace{0.05em}

\noindent and it is parametrized by $a,b,c,d$ as above, but with no restriction on $a,b,c,d$.
In the definition of the cone $\widehat{\textup{SO}(3, \CC)}$, note $\lambda = 0$ is possible; it corresponds to $a^2+b^2 + c^2 + d^2 = 0$, or to $e^2 + f^2 +g^2 +h^2 = 0$.  By the first remark after Definition \ref{defn:main}, we are free to scale cameras $B$ and $C$ so that their left $3 \times 3$ submatrices satisfy $R_{2} \in \textup{SO}(3,\CC)^{\boldsymbol{\alpha}}$ and $R_{3} \in \textup{SO}(3,\CC)^{\boldsymbol{\beta}}$, and for our formulation here we do so.  Finally, for $\CC^5$ in the source of the parametrization of $\mathcal{T}_{\textup{cal}}$, write $t_{2} = \begin{bmatrix} t_{2,1} & t_{2,2} & 1 \end{bmatrix}^{T}$ and
$t_{3} = \begin{bmatrix} t_{3,1} & t_{3,2} & t_{3,3} \end{bmatrix}^{T}$.  

At this point, we have replaced the dominant, generically injective map $\textup{SO}(3, \CC)^{\times 2} \times \CC^{5} \dashrightarrow \mathcal{T}_{\textup{cal}}$ by the dominant, generically injective parametrization \mbox{$\textup{SO}(3, \CC)^{\boldsymbol{\alpha}} \times \textup{SO}(3, \CC)^{\boldsymbol{\beta}} \times \CC^{5} \dashrightarrow \mathcal{T}_{\textup{cal}}$}.  Also, we have injective, dominant maps $V(\alpha_{1}a + \alpha_{2}b + \alpha_{3}c + \alpha_{4}d - 1) \rightarrow \textup{SO}(3, \CC)^{\boldsymbol{\alpha}}$ and $V(\beta_{1}e + \beta_{2}f + \beta_{3}g + \beta_{4}h - 1) \rightarrow \textup{SO}(3, \CC)^{\boldsymbol{\beta}}$.
Composing gives the generically 1-to-1, dominant
$V(\alpha_{1}a + \alpha_{2}b + \alpha_{3}c + \alpha_{4}d - 1) \times V(\beta_{1}e + \beta_{2}f + \beta_{3}g + \beta_{4}h - 1) \times \CC^5 \dashrightarrow  \mathcal{T}_{\textup{cal}}$.
With exactly this parametrization of $\mathcal{T}_{\textup{cal}}$, it will be most convenient to perform numerical algebraic geometry calculations.
Hence, here to compute $\textup{deg}(\mathcal{T}_{\textup{cal}}) = \#(\mathcal{T}_{\textup{cal}} \cap \textbf{L})$, we consider the square polynomial system:

\begin{itemize}
\item \textit{in 13 variables}: $a,b,c,d,e,f,g,h,t_{2,1}, t_{2,2}, t_{3,1}, t_{3,2}, t_{3,3} \in \CC$;
\item \textit{with 13 equations}: the 11 cubics (\ref{eqn:linearforms}) and 2 linear equations (\ref{eqn:firstPatch}), (\ref{eqn:secondPatch}).
\end{itemize}

\noindent The solution set equals the preimage of $\mathcal{T}_{\textup{cal}} \cap \textbf{L}$.
This system is expected to have $\textup{deg}(\mathcal{T}_{\textup{cal}})$ many solutions. 
We can solve zero-dimensional square systems of this size (in floating-point) using the 
\texttt{UseRegeneration:1} setting in \texttt{Bertini}.  That employs the \textit{regeneration} solving technique
from \cite{HSW}.  For the present system, overall, \texttt{Bertini} tracks 74,667 paths in 1.5 hours on a standard laptop 
computer to find 4912 solutions. Numerical path-tracking in \texttt{Bertini} is based on a \textit{predictor-corrector} approach.  Prediction by default
is done by the Runge-Kutta $4^{\textup{th}}$ order method; correction is by Newton steps. For more information, see \cite[Section 2.2]{BHSWBook}.
Here, this provides strong numerical evidence for
the conclusion that $\textup{deg}(\mathcal{T}_{\textup{cal}}) = 4912$.
Up to the numerical accuracy of \texttt{Bertini} and the reliability of our random number
generator used to choose $\textbf{L}$, this computation is correct with probability 1.
Practically speaking, 4912 is correct only with very high probability.

As a check for 4912, we apply the \textit{trace test} from \cite{SVW}, \cite{HR} and \cite{LS}.
A random linear form $\ell'$ on $\PP(\CC^{3 \times 3 \times 3})$ is fixed.
For $s \in \CC$, we set $L_{s} := V(\ell_{1}+s\ell', \ldots, \ell_{11} + s\ell')$,
so $L_{0} = \textbf{L}$.  Varying $s \in \CC$, the intersection $\mathcal{T}_{\textup{cal}} \cap L_{s}$
consists of $\textup{deg}(\mathcal{T}_{\textup{cal}})$ many complex \textit{paths}.
Let $T_{s} \subset \mathcal{T}_{\textup{cal}} \cap L_{s} $ be a subset of paths.
Then the trace test implies (for generic $\ell', \ell_{i}$) 
that $T_{s} = \mathcal{T}_{\textup{cal}} \cap L_{s}$ if and only if
the \textit{centroid} of $T_{s}$ computed in a consistent affine chart $\CC^{26}$, i.e.

\vspace{-0.5em}

$$
\textup{cen}(T_{s})  := {\small{\frac{1}{\# T_{s}}}}  \sum_{p_s \in T_{s}} p_s ,
$$
is an affine linear
function of $s$.  Here, we set $T_{0}$ to be the 4912 intersection points
found above.  Then we calculate $T_{1}$ with the \texttt{UserHomotopy:1} setting in \texttt{Bertini}, where the variables are $a, \ldots t_{3,3}$, and the start points are the preimages of $T_{0}$.  After this homotopy in parameter space,  $T_{1}$ is obtained by evaluating the endpoints of the track via \texttt{TrackType:-4}.
Similarly, $T_{-1}$ is computed.  Then we calculate that the following quantity
in $\CC^{26}$:

\vspace{-0.7em}

$${\big{(}}\textup{cen}(T_{1}) - \textup{cen}(T_{0}){\big{)}} - {\big{(}}\textup{cen}(T_{0}) - \textup{cen}(T_{-1}){\big{)}}$$

\vspace{0.3em}

\noindent is indeed numerically 0.  This trace test is a further verification of $4912$.
\end{proof}

\begin{rem}
In the proof of Theorem \ref{thm:4912}, when we select one point per fiber per member of 
$\mathcal{T}_{\textup{cal}} \cap \textbf{L}$, we obtain a \textit{pseudo-witness set} $\mathcal{W}$ for $\mathcal{T_{\textup{cal}}}$.  This is the fundamental data structure in numerical algebraic geometry
for computing with parameterized varieties (see \cite{HS}).
Precisely, here it is the quadruple:
\begin{itemize}
\setlength\itemsep{0.3em}
\item the \textit{parameter space} $\mathcal{P} \subset \CC^{13}$, where $\CC^{13}$ has coordinates $a, \ldots, t_{3,3}$
and $\mathcal{P} = V(\alpha_{1}a + \alpha_{2}b + \alpha_{3}c + \alpha_{4}d - 1, 
\beta_{1}e + \beta_{2}f + \beta_{3}g + \beta_{4}h - 1)$
\item the \textit{dominant map} $\Phi: \mathcal{P} \dashrightarrow \mathcal{T}_{\textup{cal}}$ in the proof of Theorem \ref{thm:4912}, e.g.
$\Phi_{1,1,1} = -2 b c t_{2,1}-2 a d t_{2,1}+a^{2} t_{2,2}+b^{2} t_{2,2}-c^{2} t_{2,2}-d^{2} t_{2,2}$
\item the \textit{generic complimentary linear space} $\textbf{L} = V(\ell_1, \ldots, \ell_{11}) \subset \PP(\CC^{3 \times 3 \times 3})$
\item the \textit{finite set} $\mathcal{W} \subseteq \mathcal{P} \subset \CC^{13}$, mapping bijectively to $\mathcal{T}_{\textup{cal}} \cap \textbf{L}$.
\end{itemize}
We heavily use this representation of $\mathcal{T}_{\textup{cal}}$ for the computations in Section \ref{sec:6}.
\end{rem}

Now, we re-visit Proposition \ref{prop:lineartrifocal}.  When $T_{A,B,C}$ is unknown but the point/line correspondence is known, the constraints there
amount to \ul{special linear slices} of $\mathcal{T}$ and of the subvariety $\mathcal{T}_{\textup{cal}}$.  
The next theorem may help the reader appreciate the specialness of these linear sections of $\mathcal{T}_{\textup{cal}}$;
in general, the intersections are not irreducible, equidimensional, nor dimensionally transverse.

\begin{thm} \label{thm:codim}
Fix generic points $x, x', x'' \in \PP^{2}$ and generic lines $\ell, \ell', \ell'' \in (\PP^{2})^{\vee}$.
In the cases of Proposition \textup{\ref{prop:lineartrifocal}}, we have the following codimensions\textup{:}

\vspace{0.3em}

\begin{itemize}
\setlength\itemsep{0.8em}
\item $L = \{T \in \PP(\CC^{3 \times 3 \times3}) : \, T(x, \ell', \ell'') = 0 \}$ is a hyperplane and $\mathcal{T}_{\textup{cal}} \cap L$ consists of one irreducible component of codimension \textup{1} in $\mathcal{T}_{\textup{cal}}$ \hfill \textup{\textbf{[PLL]}}
\item $L = \{T \in \PP(\CC^{3 \times 3 \times 3}) : \,  [\ell]_{\times}T( \--, \ell', \ell'') = 0 \}$ is a codimension \textup{2} subspace and $ \mathcal{T}_{\textup{cal}} \cap L$ consists of two irreducible components both of codimension \textup{2} in $\mathcal{T}_{\textup{cal}}$ \hfill \textup{\textbf{[LLL]}}
\item $L = \{T \in \PP(\CC^{3 \times 3 \times 3}): \,  [x'']_{\times}T( x, \ell', \--) = 0 \}$ is a codimension \textup{2} subspace and $\mathcal{T}_{\textup{cal}} \cap L$ consists of two irreducible components both of codimension \textup{2} in $\mathcal{T}_{\textup{cal}}$ \hfill\textup{\textbf{[PLP]}}
\item$L = \{T \in \PP(\CC^{3 \times 3 \times 3}) : \,  [x']_{\times}T( x, \--, \ell'') = 0 \}$ is a codimension \textup{2} subspace and $\mathcal{T}_{\textup{cal}} \cap L$ consists of two irreducible components both of codimension \textup{2} in $\mathcal{T}_{\textup{cal}}$ \hfill\textup{\textbf{[PPL]}}
\item $L = \{T \in \PP(\CC^{3 \times 3 \times 3}) : \,   [x'']_{\times}T(x, \--, \--)[x']_{\times}=0 \}$ is a codimension \textup{4} subspace and $\mathcal{T}_{\textup{cal}} \cap L$ consists of five irreducible components, one of codimension \textup{3} and four of codimension \textup{4} in $\mathcal{T}_{\textup{cal}}$. \hfill \textup{\textbf{[PPP]}}
\end{itemize}

\end{thm}

\begin{proof}
[Computational Proof]  The statements about the subspaces may shown symbolically. In the case of `\textit{LLL}', e.g., work in the ring $\QQ[\ell_{0}, \ldots, \ell''_{2}]$ with 8 variables, and write the constraint on $T \in \PP(\CC^{3 \times 3 \times 3})$ as the vanishing of 
a $ 3 \times 27$ matrix times a vectorization of $T$.  Now we check that all of the $3 \times 3$ minors of that long matrix are identically 0, but not so for $2 \times 2$ minors.

For the statements about $\mathcal{T}_{\textup{cal}} \cap L$, we offer a probability 1, numerical argument.  By \cite[Theorem A.14.10]{SW} and the discussion on page 348 about generic irreducible decompositions,
we can fix random floating-point coordinates for $x, x', x'', \ell, \ell', \ell''$.  With the parametrization $\Phi$ of $\mathcal{T}_{\textup{cal}}$ from the proof of Theorem \ref{thm:4912}, the \texttt{TrackType:1} setting 
in \texttt{Bertini} is used to compute a \textit{numerical irreducible decomposition} for the \textit{preimage} of $\mathcal{T}_{\textup{cal}} \cap L$ per each case.  That outputs a 
witness set, i.e. general linear section, per irreducible component.  \texttt{Bertini}'s \texttt{TrackType:1} is based on regeneration, monodromy and the trace test; see \cite[Chapter 15]{SW} or
\cite[Chapter 8]{BHSWBook} for a description.  

Here, the `\textit{PPP}' case is most subtle since the subspace $L \subseteq \PP(\CC^{3 \times 3 \times 3})$ 
is codimension 4, but the linear section $\mathcal{T}_{\textup{cal}} \cap L \subseteq \mathcal{T}_{\textup{cal}}$ includes a codimension 3 component.
The numerical irreducible decomposition above consists of five components of dimensions $8, 7, 7, 7, 7$ in $a, \ldots, t_{3,3}$-parameter space.
Thus, it suffices to verify that the map to $\mathcal{T}_{\textup{cal}}$ is generically injective restricted to the union of these components.
For that, we take one general point on each component from the witness sets, and test whether that point satisfies $a^2 + b^2 + c^2 + d^2 \not\approx 0$ and $e^2 + f^2 + g^2 + h^2 \not\approx 0$.  This indeed holds for all components.  Then, we test using singular value decomposition (see \cite[Theorem 3.2]{Dem}) whether the point maps to a camera triple with linearly independent centers. 
Linear independence indeed holds for all components.  From Theorem \ref{thm:calibratedConfig}, the above parametrization is generically injective on this locus.  Hence the image $\mathcal{T}_{\textup{cal}} \cap L$ consists of distinct components with the same dimensions $8,7,7,7,7$.
This finishes `\textit{PPP}'.  The other cases are similar.
\end{proof}

Mimicking the proof of Proposition \ref{prop:finite}, and using the `top-dimensional'
clause in Lemma \ref{lem:topdim}, we can establish the following finiteness result for $\mathcal{T}_{\textup{cal}}$:

\begin{lem} \label{lem:trifocalFinite}
For each problem in Theorem \textup{\ref{thm:main}}, given generic image correspondence data, there are only 
finitely many tensors $T \in \mathcal{T}_{\textup{cal}}$ that satisfy all of the linear conditions from Proposition \textup{\ref{prop:lineartrifocal}}.
\end{lem}

%
%
%
%

We have arrived at a relaxation for each minimal problem in
Theorem \ref{thm:main}, as promised. Namely, for a problem there we can fix a random instance
of image data, and
we seek those calibrated trifocal tensors that satisfy the -- merely necessary -- linear conditions in \ref{prop:lineartrifocal}.
Geometrically, this is equivalent intersect the special linear sections of $\mathcal{T}_{\textup{cal}}$ from Theorem \ref{thm:codim}.
In Section \ref{sec:6},
we will use the pseudo-witness set representation $(\mathcal{P}, \Phi, \textbf{L}, \mathcal{W})$ of $\mathcal{T}_{\textup{cal}}$ from Theorem \ref{thm:4912}
to compute these special slices of $\mathcal{T}_{\textup{cal}}$ in $\texttt{Bertini}$.  Conveniently, \texttt{Bertini} outputs a calibrated camera triple per calibrated trifocal tensor in the intersection; this is because all solving is done in the parameter space $\mathcal{P}$, or in other words, camera space. To solve the original minimal problem, we then
test these configurations against the tight conditions of
Theorem \ref{thm:multi-view}.

\section{Proof of Main Result} \label{sec:6}

In this section, we put all the pieces together and we determine the algebraic degrees of the minimal problems 
in Theorem \ref{thm:main}.  Mathematically, these degrees represent interesting \underline{enumerative geometry} problems; in vision, related work for three
\textit{uncalibrated} views appeared in \cite{OZA}.  The authors
considered correspondences `\textit{PPP}' and `\textit{LLL}' and they determined 3 degrees for projective (uncalibrated) views, using the larger group actions
present in that case.
Here, all 66 degrees for calibrated views in Theorem \ref{thm:main} are new.

Now, recall from Proposition \ref{prop:finite} that solutions $(A,B,C)$
to the problems in Theorem \ref{thm:main} in particular must have non-identical centers.  So, by the second remark after Definition \ref{defn:main}, they associate to nonzero tensors $T_{A,B,C}$, and thus to 
well-defined points in the projective variety $\mathcal{T}_{\textup{cal}}$.  Conversely, however, there are special subloci of $\mathcal{T}_{\textup{cal}}$ that are not physical.
Points in these subvarieties (introduced next) are extraneous to Theorem \ref{thm:main}, 
because they correspond to configurations with a $3 \times 4$ matrix whose left $3 \times 3$ submatrix $R$ is not a rotation, but instead
satisfies $RR^{T} = R^{T}R = 0$.

\begin{defnprop} \label{defnprop:specialLoci}
Recall the parametrization of $\mathcal{T}_{\textup{cal}}$ by $a, \ldots, t_{3,3}$ from 
Theorem \textup{\ref{thm:4912}}.  Let $\mathcal{T}_{\textup{cal}}^{0,1} \subset \mathcal{T}_{\textup{cal}}$
be the closure of the image of the rational map restricted to the locus $a^2+b^2+c^2+d^2=0$.
Let $\mathcal{T}_{\textup{cal}}^{1, 0} \subset \mathcal{T}_{\textup{cal}}$
be the closure of the image of the rational map restricted to the locus $e^2+f^2+g^2+h^2=0$.
Let $\mathcal{T}_{\textup{cal}}^{0, 0} \subset \mathcal{T}_{\textup{cal}}$
be the closure of the image of the rational map restricted to the locus $a^2+b^2+c^2+d^2=0$ and $e^2+f^2+g^2+h^2=0$.
Then these subvarieties are irreducible with\textup{:}
$\dim(\mathcal{T}_{\textup{cal}}^{0,0}) = 9$ and $\textup{deg}(\mathcal{T}_{\textup{cal}}^{0,0})=1296$\textup{;}
 $\dim(\mathcal{T}_{\textup{cal}}^{0,1}) = 10$ and $\textup{deg}(\mathcal{T}_{\textup{cal}}^{0,1})=2616$\textup{;}
 $\dim(\mathcal{T}_{\textup{cal}}^{1,0}) = 10$ and $\textup{deg}(\mathcal{T}_{\textup{cal}}^{1,0})=2616$.

\end{defnprop}

\begin{proof}[Computational Proof]
The restricted parameter spaces: 
\begin{align*}
\mathcal{P} \cap V(a^2+b^2+c^2+d^2), \,\,\, \mathcal{P} \cap V(e^2+f^2+g^2+h^2), \\
\mathcal{P} \cap V(a^2+b^2+c^2+d^2, \,e^2+f^2+g^2+h^2) \subset \CC^{13},
\end{align*}

\vspace{-0.1em}

\noindent where 
$\mathcal{P} = V(\alpha_{1}a + \alpha_{2}b + \alpha_{3}c + \alpha_{4}d - 1, \, \beta_{1}e+ \beta_{2}f + \beta_{3}g + \beta_{4}h -1)$,
are irreducible, therefore their images $\mathcal{T}_{\textup{cal}}^{0,1}, \mathcal{T}_{\textup{cal}}^{1,0}, \mathcal{T}_{\textup{cal}}^{0,0} \subset \PP(\CC^{3 \times 3 \times 3})$
are irreducible.
The dimension statements are verified by picking a random point in the restricted parameter spaces,
and then by computing the rank of the derivative of the restricted rational map $\Phi$ at that point.  This 
rank equals the dimension of the image with probability 1, by generic smoothness over $\CC$ \cite[III.10.5]{Hartshorne} and the preceding \cite[III.10.4]{Hartshorne}.  For the degree statements, the approach from Theorem \ref{thm:4912} may be used.
For $\mathcal{T}_{\textup{cal}}^{0,1}$ we fix a random linear subspace $\textbf{M} \subset \PP(\CC^{3 \times 3 \times 3})$ of complementary dimension, i.e $\dim(\textbf{M}) = 16$, so
$\textup{deg}(\mathcal{T}_{\textup{cal}}^{0,1}) = \#(\mathcal{T}_{\textup{cal}}^{0,1} \cap \textbf{M})$.
We pull back to $\mathcal{P} \cap V(a^2+b^2+c^2+d^2) \cap \Phi^{-1}(\textbf{M})$, and use
the \texttt{UseRegeneration:1} setting in \texttt{Bertini} to solve for this.
This run outputs 2616 floating-point tuples in $a, \ldots, t_{3,3}$ coordinates.
Then, we apply the parametrization $\Phi$ and check that the image of these are 2616 numerically 
distinct tensors, i.e. the restriction
$\Phi |_{\mathcal{P} \cap V(a^2+b^2+c^2+d^2)}$ is generically injective.
It follows that $\textup{deg}(\mathcal{T}_{\textup{cal}}^{0,1}) = 2616$, up to numerical accuracy
and random choices.  To verify this degree further, we apply the trace test as in Theorem \ref{thm:4912}, and this finishes the computation for $\textup{deg}(\mathcal{T}_{\textup{cal}}^{0,1})$.
Since $\mathcal{T}_{\textup{cal}}^{0,1}$ and $\mathcal{T}_{\textup{cal}}^{0,1}$ are linearly isomorphic under the permutation $T_{ijk} \mapsto T_{ikj}$, this implies
$\textup{deg}(\mathcal{T}_{\textup{cal}}^{1,0}) = 2616$.  The computation for $\textup{deg}(\mathcal{T}_{\textup{cal}}^{0,0})$ is similar.
\end{proof}

Now, we come to the proof of Theorem \ref{thm:main}, at last.  The outline was given in the last paragraph of Section \ref{sec:5}: for computations,
solving the polynomial systems of multi-view equations (see Theorem \ref{thm:multi-view}) is relaxed to taking a special linear section of the calibrated
trifocal variety $\mathcal{T}_{\textup{cal}}$ (see Theorem \ref{thm:codim}).  Then, to take this slice,
we use the numerical algebraic geometry technique of \textit{coefficient-parameter homotopy} \cite[Theorems 7.1.1, A.13.1]{SW}, i.e. 
a general linear section is moved in a homotopy to the special linear section.  

\begin{proof}[Computational Proof of Theorem \textup{\ref{thm:main}}]
Consider one of the problems `$w_{1} \textit{PPP} + w_{2} \textit{PPL} + w_{3} \textit{PLP} + w_{4} \textit{LLL} + w_{5} \textit{PLL}$'
in Theorem \ref{thm:main}, so that the weights $(w_{1}, w_{2}, w_{3}, w_{4}, w_{5}) \in \ZZ_{\geq 0}^{5}$ satisfy $3w_{1} + 2w_{2} + 2w_{3} + 2w_{4} + w_{5} = 11$ and $w_{2} \geq w_{3}$.  Fix one general instance of this problem, by taking image data with random floating-point coordinates.  Each point/line image correspondence
in this instance defines a special linear subspace of $\PP(\CC^{3 \times 3 \times 3})$, as in Theorem \ref{thm:codim}.  The intersection of these 
is one subspace $L_{\textup{special}}$ expressed in floating-point; 
using singular value decomposition, we verify that its codimension in $\PP(\CC^{3 \times 3 \times 3})$ is the expected 
$4w_{1} + 2w_{2} + 2w_{3} + 2w_{4} + w_{5} = 11 + w_{1}$.
By Proposition \ref{prop:lineartrifocal}, $L_{\textup{special}}$ represents necessary conditions for consistency, so we seek $ \mathcal{T}_{\textup{cal}} \cap L_{\textup{special}}$.  If $w_{1} > 0$, then this intersection is not dimensionally transverse by
the `\textit{PPP}' clause of Theorem \ref{thm:codim}.  To deal with a square polynomial system, we fix a general linear space ${L'}_{\textup{special}} \supseteq {L}_{\textup{special}}$ of codimension 11 in $\PP(\CC^{3 \times 3 \times 3})$ and now seek $\mathcal{T}_{\textup{cal}} \cap {L'}_{\textup{special}}$.  This step is known as \textit{randomization} \cite[Section 13.5]{SW} in numerical algebraic geometry, and it is needed to apply the parameter homotopy result \cite[Theorem 7.1.1]{SW}.

The linear section $\mathcal{T}_{\textup{cal}} \cap {L'}_{\textup{special}}$ is found numerically by a degeneration.
In the proof of Theorem \ref{thm:4912}, we computed a pseudo-witness set for $\mathcal{T}_{\textup{cal}}$.
This includes a general complimentary linear section $\mathcal{T}_{\textup{cal}} \cap \textbf{L}$, and the preimage 
$\Phi^{-1}(\mathcal{T}_{\textup{cal}} \cap \textbf{L})$ of
$\textup{deg}(\mathcal{T}_{\textup{cal}})=4912$ points in $a, \ldots, t_{3,3}$ space.
Writing $\textbf{L} = V(\ell_{1}, \ldots, \ell_{11})$ and $L'_{\textup{special}} = V(\ell'_{1}, \ldots, \ell'_{11})$
for linear forms $\ell_{i}$ and $\ell'_{i}$ on $\PP(\CC^{3 \times 3 \times 3})$, consider the following 
homotopy function $H : \CC^{13} \times \RR \rightarrow \CC^{13}$:

\vspace{-0.7em}

$$
H(a, \ldots, t_{3,3}, s) :=
\begin{bmatrix}
s\cdot \ell_{1}\big{(}\Phi(a, \ldots, t_{3,3})\big{)} + (1-s) \cdot \ell'_{1}\big{(}\Phi(a, \ldots, t_{3,3})\big{)}\\[4pt]
\vdots \\[4pt]
s\cdot \ell_{11}\big{(}\Phi(a, \ldots, t_{3,3})\big{)} + (1-s) \cdot \ell'_{11}\big{(}\Phi(a, \ldots, t_{3,3})\big{)}\\[4pt]
\alpha_{1}a+\alpha_{2}b+\alpha_{3}c+\alpha_{4}d-1\\[4pt]
\beta_{1}e+\beta_{2}f+\beta_{3}g+\beta_{4}h-1
\end{bmatrix}.
$$

\vspace{0.3em}

\noindent Here $s \in \RR$ is the \textit{path variable}.  As $s$ moves from 1 to 0, $H$
defines a family of square polynomial systems in the 13 variables $a, \ldots, t_{3,3}$.
The \textit{start system} $H(a, \ldots, t_{3,3}, 1) = 0$ has solution set $\Phi^{-1}(\mathcal{T}_{\textup{cal}} \cap \textbf{L})$
and the \textit{target system} $H(a, \ldots, t_{3,3}, 0) = 0$ has solution set $\Phi^{-1}(\mathcal{T}_{\textup{cal}} \cap L'_{\textup{special}}).$  With the \texttt{UserHomotopy:1} setting in \texttt{Bertini}, we track the 4912
solution paths from the start to target system.  By genericity of $\textbf{L}$ in the start system, these solution paths are smooth \cite[Theorem 7.1.1(4), Lemma 7.1.2]{SW}.
The finite endpoints of this track consist of solutions to the target system.  By the principle of 
coefficient-parameter homotopy \cite[Theorem A.13.1]{SW},
every isolated point in $\Phi^{-1}(\mathcal{T}_{\textup{cal}} \cap L'_{\textup{special}})$ is an endpoint, with probability 1.
Note that in general, coefficient-parameter homotopy -- i.e., the tracking of solutions of a \textit{general} instance of a parametric system of equations to solutions of a \textit{special} instance -- may be used to find all \textit{isolated} solutions to \textit{square} polynomial systems. 
Here, by Lemma \ref{lem:trifocalFinite}, $\mathcal{T}_{\textup{cal}} \cap L_{\textup{special}}$ is a scheme with finitely many points. By Bertini's theorem \cite[Theorem 13.5.1(1)]{SW}, $\mathcal{T}_{\textup{cal}} \cap L'_{\textup{special}}$ also consists of finitely many points, using genericity of $L'_{\textup{special}}$. On the other hand, by Proposition \ref{prop:finite}, all solutions $(A,B,C)$ to the instance of the original minimal problem indexed by $w \in \ZZ^5_{\geq 0}$
have linearly independent centers in $\PP^{3}$.  Moreover, a configuration $(A,B,C)$ with linearly independent centers is an isolated 
point in $\Phi^{-1}(T_{A,B,C})$, thanks to Theorem \ref{thm:calibratedConfig}.  Therefore, it follows that all solutions to the problem from Theorem \ref{thm:main}
are among the isolated points in $\Phi^{-1}(\mathcal{T}_{\textup{cal}} \cap L'_{\textup{special}})$, and so the endpoints of the above homotopy.  

For each minimal problem in Theorem \ref{thm:main}, after the above homotopy, \texttt{Bertini} returns 4912 finite endpoints in $a, \ldots, t_{3,3}$ space.  We pick out which of these endpoints are solutions to the original minimal problem by performing a sequence of checks, as explained next.  First of all, of these endpoints, let us keep only those that lie in $\Phi^{-1}(\mathcal{T}_{\textup{cal}} \cap L_{\textup{special}})$, as opposed to those that lie
just in the squared-up target solution set $\Phi^{-1}(\mathcal{T}_{\textup{cal}} \cap L'_{\textup{special}})$. Second, we remove points that satisfy $a^2 + b^2 + c^2 + d^2 \approx 0$ or $e^2 + f^2 + g^2 + h^2 \approx 0$, because they are non-physical (see Definition/Proposition \ref{defnprop:specialLoci}).
Third, we verify that, in fact, all remaining points correspond to camera configurations $(A,B,C)$ with linearly independent centers.  
This means that the equations in Theorem \ref{thm:multi-view} generate the multi-view ideals (recall Definition \ref{defn:multi-view}).  Fourth, we check which remaining points satisfy 
those tight multi-view equations. To test this robustly in floating-point, note that the equations in Theorem \ref{thm:multi-view} are equivalent to rank drops of the concatenated matrices there, hence we test for those rank drops using singular value decomposition.  If the ratio of two consecutive singular values exceeds $10^5$, then this is taken as an indication that all singular values below are numerically 0, thus the matrix drops rank.
Fifth, and conversely, we verify that all remaining configurations $(A,B,C)$ avoid epipoles (recall Definition \ref{defn:epi}) for the fixed random instance of image correspondence data, so the converse Lemma \ref{lem:avoid} applies to prove consistency.  Lastly, we verify that all solutions are numerically distinct.  Ultimately, the output of this procedure is a list of all calibrated camera configurations over $\CC$ that are solutions to the fixed random instances of the minimal problems, where these solutions are expressed in floating-point and $a, \ldots, t_{3,3}$ coordinates.  The~numbers~of~solutions~are~the~algebraic~degrees~from~Theorem~\ref{thm:main}.

As a check for this numerical computation, we repeat the entire calculation for other random instances of correspondence data.  For each minimal problem, we obtain the same algebraic degree each time.  One instance per problem solved to high precision is provided on this paper's webpage.
\end{proof}

\begin{exmp}
We illustrate the proof of Theorem \ref{thm:main} by returning to the \mbox{instance}
of `$1 PPP + 4 \textit{PPL}$' in Example \ref{exmp:160}.  Here 
$L_{\textup{special}} \subset \PP(\CC^{3\times 3\times 3})$ formed by intersecting 
subspaces from Theorem \ref{thm:codim} is codimension 12, hence 
$L'_{\textup{special}} \varsupsetneq L_{\textup{special}}$.  Tracking $\textup{deg}(\mathcal{T}_{\textup{cal}})$ many
points in the pseudo-witness set $\Phi^{-1}(\mathcal{T}_{\textup{cal}} \cap \textbf{L})$ to the target $\Phi^{-1}(\mathcal{T}_{\textup{cal}} \cap L'_{\textup{special}})$, we get 4912 finite endpoints.
Testing membership in $L_{\textup{special}}$,
we get 2552 points in $\Phi^{-1}(\mathcal{T}_{\textup{cal}} \cap L_{\textup{special}})$.
Among these, 888 points satisfy $a^2 + b^2 + c^2 + d^2 \approx 0$, so they are non-physical 
(corresponding to $3 \times 4$ matrices 
with left submatrices that are not rotations).
The remaining 1664 points turn out to correspond to calibrated camera configurations 
with linearly independent centers.  Checking satisfaction of the equations from Theorem \ref{thm:multi-view}, we end up with 
$160$ solutions.
\end{exmp}

\begin{rem}
The proof of Theorem \ref{thm:main} is constructive.  From the solved random instances, 
one may build solvers for each minimal problem, using coefficient-parameter homotopy.  Here
the start system is the solved instance of the minimal problem and the target system
is another given instance.  Such a solver is \underline{optimal}
in the sense that the number of paths tracked equals the true algebraic degree of the problem.
Implementation is left to future work.
\end{rem}

\begin{rem}
All degrees in Theorem \ref{thm:main} are divisible by 8.  We would like to understand why.
What are the \textit{Galois groups} \cite{HRS} for these minimal problems?
\end{rem}

\begin{rem}
Practically speaking, given image correspondence data defined over $\RR$, 
\underline{only real solutions} $(A,B,C)$ to the minimal problems in Theorem \ref{thm:main}
are of interest to RANSAC-style $3D$ reconstruction algorithms.  Does there exist image data 
such that all solutions are real?  Also, for the image data observed in practice,
what is the distribution of the number of real solutions?
\end{rem}

\bigskip \bigskip \bigskip

\footnotesize 

\noindent \textbf{Author's address:}

\noindent Joe Kileel, University of California, Berkeley, USA,
\texttt{jkileel@math.berkeley.edu}


\bigskip
\medskip


\begin{thebibliography}{10}

\bibitem{ASSSS} 
S.~Agarwal, N.~Snavely, I.~Simon, S.M.~Seitz and R.~Szeliski,
Building Rome in a day,
{\em Proc. Int. Conf. on Comput. Vision} (2009), 72--79.

\bibitem{AO}
C.~Aholt and L.~Oeding,
The ideal of the trifocal variety,
{\em Math. Comput.} {\bf 83} (2014) 2553--2574,
\arxiv{1205.3776}.

\bibitem{AST}
C.~Aholt, B.~Sturmfels and R.~Thomas,
A Hilbert scheme in computer vision,
{\em Can. J. Math.} {\bf 65} (2013) 961--988,
\arxiv{1107.2875}.

\bibitem{AT}
A.~Alzati and A.~Tortora, 
A geometric approach to the trifocal tensor, 
{\em J. Math. Imaging Vision} \textbf{38} (2010), no. 3, 159--170.

\bibitem{BHSWSoftware}
D.J.~Bates, J.D.~Hauenstein, A.J.~Sommese and C.W.~Wampler, 
{\em Bertini: Software for numerical algebraic geometry}. 
Available at \url{http://bertini.nd.edu}.

\bibitem{BHSWBook}
\bysame,
{\em Numerically solving polynomial systems with Bertini}, Software, Environments, and Tools \textbf{25},
Society for Industrial and Applied Mathematics, Philadelphia, PA, 2013.

\bibitem{Dem} 
J.~Demmel,
{\em Applied numerical linear algebra},
Society for Industrial and Applied Mathematics,
Philadelphia, PA, 1997.

\bibitem{Eis}
D.~Eisenbud, 
{\em Commutative algebra: with a view toward algebraic geometry},
Graduate Texts in Mathematics \textbf{150},
Springer-Verlag, New York, 1995.

\bibitem{FB} 
M.~Fischler and R.~Bolles, 
Random sample consensus: a paradigm for model fitting with application to 
image analysis and automated cartography, 
{\em Commun. Assoc. Comp. Mach.} \textbf{24} (1981) 381--395.

\bibitem{GS} 
D.~Grayson and M.~Stillman,
{\em Macaulay2, a software system for research in algebraic geometry}. 
Available at  \url{http://www.math.uiuc.edu/Macaulay2/}.

\bibitem{Har}
R.I.~Hartley,
Lines and points in three views and the trifocal tensor,
{\em Int. J. Comput. Vision} \textbf{22} (1997), no. 2, 125--140.

\bibitem{Hartshorne}
R. Hartshorne, 
{\em Algebraic geometry}, Graduate Texts in Mathematics \textbf{52}, 
Springer-Verlag, New York, 1977.

\bibitem{HZ}
R.I.~Hartley and A.~Zisserman,
{\em Multiple view geometry in computer vision},
Cambridge University Press, 2003.

\bibitem{HR}
J.D.~Hauenstein and J.I.~Rodriguez,
Numerical irreducible decomposition of multiprojective varieties,
\arxiv{1507.07069v2}.

\bibitem{HRS}
J.D.~Hauenstein, J.I.~Rodriguez and F.~Sottile,
Numerical computation of Galois groups,
\arxiv{1605.07806}.

\bibitem{HS}
J.D.~Hauenstein and A.J.~Sommese,
Witness sets of projections,
{\em Appl. Math. Comput.}
\textbf{217}, (2010), no. 7, 3349--3354.
     
\bibitem{HSW}
J.D.~Hauenstein, A.J.~Sommese and C.W.~Wampler,
Regeneration homotopies for solving systems of polynomials,
{\em Math. Comp.} \textbf{80}, (2011), no. 273, 345--377.
     
\bibitem{Kuk} 
Z.~Kukelova, 
{\em Algebraic methods in computer vision},
Doctoral Thesis, Czech Technical University in Prague, 2013.
     
\bibitem{Lan} 
J.M.~Landberg:
{\em Tensors: geometry and applications},
 Graduate Studies in Mathematics \textbf{128},
 American Mathematical Society, Providence, RI, 2012.

\bibitem{LS}
A.~Leykin and F.~Sottile,
Trace test,
\arxiv{1608.00540}.

\bibitem{Low}
D.G.~Lowe, 
Object recognition from local scale-invariant features, 
{\em Proc. Int. Conf. on Comput. Vision} (1999), 1150--1157.

\bibitem{Mar} 
E.~Martyushev,
On some properties of calibrated trifocal tensors,
\arxiv{1601.01467v3}.

\bibitem{Mat}
J.~Matthews,
Multi-focal tensors as invariant differential forms,
\arxiv{1610.04294}.


\bibitem{May} 
S.~Maybank, {\em Theory of reconstruction from image motion}, 
Springer, Berlin, 1993. 

\bibitem{Nis} 
D.~Nist\'{e}r,
An efficient solution to the five-point relative pose problem,
{\em IEEE Trans. Pattern Anal. Mach. Intell.} \textbf{26} (2004), no. 6, 756--770.

\bibitem{OZA}
M.~Oskarsson, A.~Zisserman and K.~$\mathring{\textup{A}}$str\"om,
Minimal projective reconstruction for combinations
of points and lines in three views, 
{\em Image Vision Comput.} \textbf{22} (2004), no. 10, 777-785.

\bibitem{SVW}
A.J.~Sommese, J.~Verschelde and C.W.~Wampler,
Symmetric functions applied to decomposing solution sets of polynomial systems,
{\em SIAM J. Numer. Anal.} \textbf{40} (2002), no. 6, 2012--2046.

\bibitem{SW}
A.J.~Sommese and C.W.~Wampler, 
{\em The numerical solution of systems of polynomials},
The Publishing Co. Pte. Ltd., Hackensack, NJ, 2005.

\bibitem{TPH} 
M.~Trager, J.~Ponce and M.~Hebert, 
Trinocular geometry revisited, 
{\em Int. J.
Comput. Vision} \textbf{120} (2016), no. 2, 134--152.

\end{thebibliography}
\end{document}